\documentclass[a4paper, 10pt, twoside]{article}

\usepackage{exscale, fullpage,url}
\usepackage[centertags]{amsmath}
\usepackage{color,amscd}
\usepackage{amssymb}
\usepackage{amsthm}
\usepackage{dsfont}
\usepackage[all]{xy}
\usepackage{lscape}
\usepackage{pdflscape}
\usepackage{algpseudocode}
\usepackage{mathrsfs}
\usepackage{upgreek}
\usepackage{tikz-cd}
\usepackage{hyperref}

\newtheorem{definition}{Definition}[section]
\newtheorem{theorem}[definition]{Theorem}
\newtheorem{proposition}[definition]{Proposition}
\newtheorem{lemma}[definition]{Lemma}
\newtheorem{corollary}[definition]{Corollary}

\newcommand{\nd}{\noindent}

\newcommand{\dG}{{\mathds G}}

\newcommand{\dR}{{\mathds R}}
\newcommand{\dC}{{\mathds C}}
\newcommand{\dQ}{{\mathds Q}}

\newcommand{\dZ}{{\mathds Z}}
\newcommand{\dP}{{\mathds P}}

\newcommand{\dA}{{\mathbb A}}

\newcommand{\cD}{\mathcal{D}}
\newcommand{\cE}{\mathcal{E}}
\newcommand{\cF}{\mathcal{F}}

\newcommand{\cH}{\mathcal{H}}
\newcommand{\cI}{\mathcal{I}}

\newcommand{\cK}{\mathcal{K}}
\newcommand{\cL}{\mathcal{L}}
\newcommand{\cM}{\mathcal{M}}
\newcommand{\cN}{\mathcal{N}}
\newcommand{\cO}{\mathcal{O}}

\newcommand{\cQ}{\mathcal{Q}}
\newcommand{\cR}{\mathcal{R}}

\newcommand{\cU}{\mathcal{U}}

\newcommand{\fa}{\mathfrak{a}}

\newcommand{\fg}{\mathfrak{g}}

\newcommand{\D}{\displaystyle}

\newcommand{\SC}{\scriptstyle}

\DeclareMathOperator{\Spec}{\textup{Spec}\,}

\DeclareMathOperator{\vol}{\textup{vol}}

\DeclareMathOperator{\im}{\textup{im}}
\DeclareMathOperator{\FL}{\textup{FL}}
\DeclareMathOperator{\id}{\textup{id}}
\DeclareMathOperator{\Gr}{\textup{Gr}}

\newcommand{\Gl}{\textup{Gl}}

\newcommand{\MHM}{\operatorname{MHM}}

\long\def\inhibe#1\endinhibe{\relax}

\newcommand{\ux}{\underline{x}}
\newcommand{\ur}{\underline{r}}
\newcommand{\upartial}{\underline{\partial}}

\newcommand{\tDD}{D_{\widetilde{V}}}

\newcommand{\tth}{\widetilde{h}}

\newcommand{\tchi}{\widetilde{\chi}}

\newcommand{\tdelta}{\widetilde{\mathcal{\delta}}}

\newcommand{\tOO}{A_{\widetilde{V}}}

\DeclareMathOperator{\SP}{\rm Sp}
\DeclareMathOperator{\U}{\mathcal{U}}
\DeclareMathOperator{\Sym}{\textup{Sym}}

\newcommand{\Lotimes}{\stackrel{\mathbf{L}}{\otimes}}

\newcommand{\Der}{\text{Der}}

\begin{document}
\title{Tautological systems and free divisors}
\author{Luis Narv\'{a}ez Macarro and Christian Sevenheck}

\maketitle

\begin{abstract}
We introduce tautological systems defined by
prehomogeneous actions of reductive algebraic groups. If the
complement of the open orbit is a linear free divisor satisfying
a certain finiteness condition, we show that these systems underly
mixed Hodge modules. A dimensional reduction is considered
and gives rise to one-dimensional differential systems generalizing the quantum differential equation of projective spaces.
\end{abstract}

\footnote{
Partially supported by MTM2016-75027-P, P12-FQM-2696, FEDER
and by the DFG project Se 1114/3-1.
}

\section{Introduction}
\label{sec:Introduction}

In this paper we consider systems of differential equations defined by certain prehomogeneous vector spaces, i.e.
actions of algebraic groups admitting an open dense orbit. These $\cD$-modules turn out to be examples of
so-called \emph{tautological systems}, as studied in \cite{HottaEq}, \cite{KapranovReduct} and more recently in a series of papers
\cite{Taut1, Taut2, Taut3}. The underlying philosophy is that in the case where the group is an algebraic torus, these
$\cD$-modules are nothing but the well-known GKZ-systems of Gelfand', Graev, Kapranov and Zelevinski (see, e.g., \cite{GKZ1, Adolphson}).
They play a prominent role in toric mirror symmetry (\cite{Giv7}, \cite{Ir2}, \cite{ReiSe, ReiSe2}), and part of our motivation for this
paper comes from the wish to understand whether certain tautological systems can potentially occur as quantum $\cD$-modules.

We consider more specifically the situation where a reductive algebraic group acts on a vector space $V$ such that there exists an
open and dense orbit. Moreover, we suppose that the complement of this orbit is a divisor, which has a reduced equation defined by the determinant of the matrix the columns of which are the coefficient
of the vector fields defined by the corresponding Lie algebra action. This is exactly the situation studied in \cite{BM,GMNS}, where the discriminant divisor is called \emph{linear free}: it is a free divisor in the sense of K.~Saito (see \cite{KS1}), that is, the sheaf of logarithmic vector fields is locally free, and it is linear free since the coefficients of these vector fields are linear functions.
One can consider the subgroup (called $A_D$ below) of the given algebraic group consisting of linear transformations stabilizing \emph{all} fibres of a reduced equation defining the divisor.
It is the group $\dG_m\times A_D$, acting on the space $\dC\times V$ that defines the tautological system, called $\cM$ in the main body of this article
(see Definition \ref{def:TautSystem}). This is a $\cD$-module on the dual space of $\dC\times V$, where the extra factor is needed to ensure some homogeneity property. From this it follows that the tautological system is regular \emph{if} it is holonomic, which may not be the case in general. Indeed, a theorem of Hotta \cite {HottaEq} shows that holonomicity holds if
the action of $G$ has finitely many orbits. In order to have this property, we restrict to the case of \emph{strongly Koszul free} divisors (see the beginning of Section \ref{sec:FreeDiv}), a notion that dates back to \cite{granger_schulze_rims_2010}. We prove in section 4 (see Theorem \ref{theo:TautSystMHM}) that for strongly Koszul linear free divisors,
the associated tautological system underlies a mixed Hodge module. More precisely, this main result can be formulated as follows.
\begin{theorem}[Compare Lemma \ref{lem:CalculationTautSystem} and Theorem \ref{theo:TautSystMHM} below]
Let $V=\dA^n$ and let $D\subset V$ be a strongly Koszul reductive linear free divisor
with (reduced) defining equation $h\in \cO_V$. Let $V^\vee$ be the dual space, and $h^\vee$ an equation for the dual divisor. Consider the (free) $\cO_V$-module of vector fields logarithmic to all fibres of $h^\vee$, that is,
$$
\theta(-\log h)=
\bigoplus_{i=1}^{n-1} \cO_V \delta^\vee_i
$$
where $\delta^\vee_i(h^\vee)=0$. Put $\widetilde{V}^\vee:=\dA^1_{\lambda_0}\times V^\vee$ with coordinates
$\lambda_0,\ldots, \lambda_n$, then for all
$\beta_0 \in \dZ$ such that
$\beta_0< \min\left(\bigcup_{k\geq 0} (k+n\cdot \{\textup{roots of }b_h(s)\})\right)$ (where $b_h(s) \in \dC[s]$ is the Bernstein-Sato polynomial of $h$) the $\cD_{\widetilde{V}}$-module (called tautological system associated to $D$ in the main body of this article)
$$
\frac{\cD_{\widetilde{V}^\vee}}{(\partial_{\lambda_0}^n-h(\partial_{\lambda_1},\ldots,\partial_{\lambda_n}),
\delta^\vee_1,\ldots,\delta^\vee_{n-1},\widetilde{\chi}^\vee+(n+1)+\beta_0)}
$$
(where $\widetilde{\chi}^\vee=\sum_{i=0}^{n} \lambda_i\partial_{\lambda_i}$)
underlies a mixed Hodge module on $\widetilde{V}^\vee$.
\end{theorem}

We will see later that the above presentation of the tautological system as well as the condition on the parameter $\beta_0$ is quite natural. A similar statement for GKZ-systems
has been shown in \cite[Theorem 3.5.]{Reich2}.

The proof of Theorem \ref{theo:TautSystMHM}  is based on two main observations: As in the
case of GKZ-$\cD$-modules (see \cite{Reich2}), regular singular tautological systems are obtained as Fourier-Laplace transforms of certain monodromic
$\cD$-module on the dual space, that is, the space $\dC\times V$ in our above notation. Using the Radon transformation formalism for $\cD$-modules,
it is sufficient to show that this Fourier-Laplace transform underlies a mixed Hodge module. This is done by expressing this module as a direct
image of a (twisted) structure sheaf, and the main point is to show that multiplication with the
coordinate corresponding to the first factor in $\dC\times V$ is invertible on that module (this is parallel to the main result of \cite{SchulWalth2}).
This is done in sections 2 and 3, based on the construction of Spencer complexes associated with some
 Lie-Rinehart-algebras.  These complexes can be filtered in such a way that their graded complexes are Koszul complexes, which become acyclic under our strongly Koszul hypothesis. This technique has been extensively used in \cite{calde_ens,calde_nar_compo,calde_nar_lct_ilc,nar_symmetry_BS}.

In the last section, we consider a dimensional reduction of the tautological systems defined by linear free divisors. This again is parallel to constructions
in toric mirror symmetry, where GKZ-systems are reduced to $\cD$-modules on the complexified K\"ahler moduli space
(see also \cite{BMW} for a more general framework).
As mentioned above, our motivation is to study potential Landau-Ginzburg models (i.e. regular functions on smooth affine varieties) that can occur in Hodge theoretic mirror
symmetry for non-toric varieties. We obtain these functions as hyperplane sections of the fibres of the equation of our free divisor. The dimensional reduction is done here using a direct image, in constrast to the toric case, where
it is a non-characteristic inverse image (see also the discussion of the example of a normal crossing divisor in section \ref{sec:DimReduc}, in particular formula \eqref{eq:InverseImageReduction}). This reflects the fact that the regular function occuring here are not Laurent polynomials, and there is in general no global coordinate system on the Milnor fibres of the free divisor (whereas Laurent polynomials are functions on algebraic tori).

Our reduced system is a $\cD$-module in two variables. It turns out that this system
(or rather its partial Fourier-Laplace transform) is isomorphic to a system already studied in detail in \cite{dGMS} and \cite{Sev1, Sev2}, where we have explicitly calculated the Gau\ss-Manin cohomology and related invariants (like the Hodge spectrum)
of hyperplane sections of the Milnor fibres of the divisor using a rather complicated algorithmic approach. Here the structure of the reduced $\cD$-modules can be directly obtained from the shape of the tautological system.
More precisely, we obtain the following statement.
\begin{theorem}[Compare Proposition \ref{prop:CalcGMSystemsLFD} below]
Let $D\subset V$ be a strongly Koszul reductive linear free divisor with defining equation $h$, let $h^\vee\in \cO_{V^\vee}$ be an equation for the dual divisor $D^\vee\subset V^\vee$. Let $p\in V\backslash D$ and write
$X:=h^{-1}(h(p))$ and $X^\vee:=(h^\vee)^{-1}(h(p))$. Put
$$
\begin{array}{rcl}
\Psi:X^\vee\times V & \longrightarrow & \dA^1_s\times \dA^1_t \\ \\
(f,x) & \longmapsto & (f(x),h(x)).
\end{array}
$$
Then we have the following expression for the
(partial localized Fourier-Laplace transformation of the) Gau\ss-Manin system of the family of hyperplane sections of the Milnor fibre $h^{-1}(p)$:
$$
\FL^{\textup{loc}}_{\dG_{m,t}}\cH^0\Psi_+\cO_{X^\vee\times V}(*(X^\vee\times D)) \cong
\frac{\D \cD_{\dA^1_z\times\dG_{m,t}}}{\D (z^n b_h(t\partial_t)-h(p)\cdot t,z^2\partial_z+ntz\partial_t)}.
$$
Here $\FL^{\textup{loc}}_{\dG_{m,t}}$ denotes the localized Fourier-Laplace transformation as defined in Formula \eqref{eq:FLloc} below.
\end{theorem}

Let us fix some notation that will be used throughout this paper. For a smooth algebraic variety over the complex numbers, we let $\cD_X$ be the
sheaf of algebraic differential operators on $X$. If $X$ is affine or $\cD$-affine, we sometimes make no distinction
between sheaves of $\cD_X$-modules and their modules of global sections.

The Fourier transformation for algebraic $\cD$-modules is used at several places
and defined as follows.
\begin{definition}\label{def:FL}
Let $Y$ be a smooth algebraic variety, $U$ be a finite-dimensional complex vector space and $U'$ its dual vector space. Denote by $\cE$ the trivial vector bundle $\tau:U\times Y\rightarrow Y$ and by $\cE'$ its dual.
Write $\text{can}:U\times U'\rightarrow\dA^1$ for the canonical morphism defined by $\text{can}(a,\varphi)=\varphi(a)$. This extends to a function $\text{can}:\cE\times\cE'\rightarrow\dA^1$. Define $\cL:=\cO_{\cE\times_{Y}\cE'}e^{-\text{can}}$, the free rank one module with differential given by the product rule. Consider also the canonical projections $p_1:\cE\times_{Y}\cE'\rightarrow\cE$, $p_2:\cE\times_{Y}\cE'\rightarrow\cE'$. The partial Fourier-Laplace transformation is then defined by
$$
\FL_{Y}:=p_{2,+}\left(p_1^+\bullet\otimes^\mathbb{L}\cL\right).
$$
\end{definition}

If the base $Y$ is a point we recover the usual Fourier-Laplace transformation and we will simply write $\FL$. Notice that although this functor is defined at the level of derived categories, it is exact, i.e., induces a functor $\FL_Y:\text{Mod}_\text{h}(\cD_\cE)\rightarrow \text{Mod}_\text{h}(\cD_{\cE'})$.

We will also need a localized version of the Fourier-Laplace transformation, defined as follows.
Suppose that $U$ is one-dimensional, with coordinate $s$.
We consider the Fourier-Laplace transformation relative  to the base $Y$ as above, and we denote the coordinate on the dual fiber $U'$ by $\tau$.  Set $z = 1 / \tau$ and denote by $j_\tau : \dG_{m, \tau} \times Y^\vee \hookrightarrow \dA^1_\tau \times Y$ and $j_z : \dG_{m, \tau} \times Y\hookrightarrow \dA^1_z \times Y = \dP^1_\tau \setminus \{\tau = 0\} \times Y$ the canonical embeddings. Let $\cN$ be an object of  $D^b(\cD_{U\times Y})$, then we put
\begin{equation}\label{eq:FLloc}
\FL^{\textup{loc}}_Y(\cN):=j_{z + }j_\tau^+ \FL_Y(\cN),
\end{equation}
notice that this functor again is exact.

\section{Lie-Rinehart algebras and Spencer complexes}

In this section we will be concerned with the following filtered rings $(R,F_\bullet)$ of differential operators:
\begin{itemize}
\item[(i)] $R =\dC[\ux][\upartial]=\dC[x_1,\dots,x_n][\partial_1,\dots,\partial_n]$ and $F_\bullet$ the usual filtration by the order of differential operators. The corresponding graded ring will be the (commutative) polynomial ring $\Gr R = \dC[\ux][\xi_1,\dots,\xi_n]$ with its usual grading:
$\dC[\ux]$ is in degree $0$ and $\xi_i = \sigma(\partial_i)$ with $\deg (\xi_i) = 1$.
\item[(ii)] $R =\dC[\ux][\upartial][s]=\dC[x_1,\dots,x_n][\partial_1,\dots,\partial_1,s]$ and $F_\bullet$ the {\em total order filtration} for which $\dC[\ux]$ is the order $0$ part and $s,\partial_1,\dots,\partial_n$ have order $1$. The corresponding graded ring will be $\Gr R = \dC[\ux][\xi_1,\dots,\xi_n,s]$ with $\dC[\ux]$ in degree $0$ and $\xi_1,\dots,\xi_n,s$ in degree $1$.
\end{itemize}
In both cases the commutative $\dC$-algebra $F_0 R$ coincides with $\dC[\ux]$ and $\dC[\ux]$ has a natural left $R$-module structure denoted by
$$ (r,f)\in R \times \dC[\ux] \longmapsto r(f) \in \dC[\ux]
$$
 (in case (ii) $s$ annihilates $\dC[\ux]$).
Moreover, any $r\in F_1 R$ can be decomposed as $r = r(1) + (r-r(1))$ and so we obtain a natural decomposition $F_1 R = (F_0 R) \oplus \left(\Gr_1 R\right)$ by identifying $r -r(1) \equiv \sigma_1(r)$.
\medskip

We have the following facts (\cite{Rinehart-1963}; see also \cite[Appendix A]{nar_symmetry_BS}):
\begin{itemize}
\item In case (i), the filtered ring $(R,F_\bullet)$ appears as the enveloping algebra of the Lie-Rinehart algebra $\Der_k(\dC[\ux]) = \bigoplus_i \left(\dC[\ux]\partial_i\right)$ over $(\dC,\dC[\ux])$ with its natural filtration.
\item In case (ii), the filtered ring $(R,F_\bullet)$ appears as the enveloping algebra of the Lie-Rinehart algebra $\left(\dC[\ux]s\right) \oplus \Der_\dC(\dC[\ux])$ over $(\dC,\dC[\ux])$ with its natural filtration. Here, the anchor map $\left(\dC[\ux]s\right) \oplus \Der_\dC(\dC[\ux]) \to \Der_\dC(\dC[\ux])$ is the projection.
\end{itemize}
Both cases are unified by the fact that $R$ appears as the enveloping algebra of the Lie-Rinehart algebra $\Gr_1 R$ over $(\dC,\dC[\ux])$.
\medskip

We will be especially interested in left $R$-modules of the form
$
\frac{R}{R \langle r_1,\dots,r_m \rangle}$
with:
\begin{itemize}
\item $r_i \in F_1 R$ for $i=1,\dots, m$ (resp. $r_1 \in F_0 R$  and $r_i \in F_1 R$ for $i=2,\dots, m$).
\item The system $\{\sigma_1(r_1), \sigma_1(r_2),\dots, \sigma_1(r_m)\}$ (resp. $\{\sigma_0(r_1) = r_1, \sigma_1(r_2),\dots, \sigma_1(r_m)\}$)
is linearly independent over $F_0 R=\dC[\ux]$.
\item The module $\bigoplus_i \left( \dC[\ux]\cdot r_i\right)$ is closed under the Lie bracket.
\end{itemize}
The above hypotheses will allow us to consider $L=\bigoplus_i \left( \dC[\ux]\cdot r_i\right)$ as a Lie-Rinehart algebra over $(\dC,\dC[\ux])$ and to take advantage of the constructions of Spencer complexes (\cite[\S 4]{Rinehart-1963}; see also \cite[(A.18)]{nar_symmetry_BS}).
\medskip

Under the above hypotheses, let us call $\U = \U(L)$ the enveloping algebra of $L$.
The {\em Cartan-Eilenberg-Chevalley-Rinehart-Spencer complex}  $\SP^\bullet_L$ associated with $L$
is defined as:
$$\textstyle \SP^{-e}_{L} = \U\otimes_{\dC[\ux]}
\bigwedge^e L ,\quad e=0,\dots, m$$ where the left $\U$-module structure comes exclusively from the left factor $\U$ of the tensor product,
the differentials
 $d^{-e}: \SP^{-e}_{L} \to \SP^{-e+1}_{L}$ are given by $ d^{{-1}}(P\otimes\lambda  ) = P\lambda$ and
\begin{eqnarray} \label{eq:diff-SP}
&\displaystyle d^{{-e}}( P\otimes(\lambda_1\wedge\cdots\wedge\lambda_e))
 =\sum_{i=1}^e
(-1)^{i-1} P\lambda_i\otimes(\lambda_1\wedge\cdots\widehat{\lambda_i}
\cdots\wedge\lambda_e)&
\end{eqnarray}
\begin{eqnarray*}
&\displaystyle + \sum_{1\leq i<j\leq e}
(-1)^{i+j}P\otimes([\lambda_i,\lambda_j]\wedge\lambda_1\wedge\cdots
\widehat{\lambda_i}\cdots\widehat{\lambda_j}
       \cdots\wedge\lambda_e), \  2\leq e\leq m,&
\end{eqnarray*}
and the augmentation  is
$ P \in \U = \SP^{0}_{L}\mapsto d^0(P):= P(1)\in \dC[\ux]$.
\medskip

Let us denote by $\overline{\SP^\bullet_L}$  the augmented complex $\SP^\bullet_L \to \dC[\ux] = \U(L)/\U(L)\langle L \rangle$.

\begin{proposition} \label{prop:res-Sp}
The complex $\SP^\bullet_L$ is a left $\U$-free resolution of $\dC[\ux]$.
\end{proposition}

\begin{proof} One has to use two ingredients. The first one is the PBW theorem which asserts that, $L$ being free over $\dC[\ux]$, the graded ring $\Gr \U$ coincides with the symmetric algebra of $L$ over $\dC[\ux]$. The second one consists of filtering $\overline{\SP^\bullet_L}$ with
$$\textstyle F_i \SP^{-e}_{L} := \left( F_{i-e}  \cU \right) \otimes_{\dC[\ux]} \bigwedge^e L,\quad F_i \dC[\ux]:= \dC[\ux], \quad i\geq 0,
$$
in such a way that the corresponding graded complex coincides with the augmented Koszul complex $\bigwedge^\bullet L \otimes_{\dC[\ux]} \Sym^\bullet L$, which is exact.
\end{proof}

From the inclusion $L=\bigoplus_i \left( \dC[\ux]\cdot r_i\right) \subset F_1 R$ we obtain a map of filtered rings $\U(L) \to R$. We define the {\em Spencer complex over $R$ associated with $\ur = (r_1,\dots,r_m)$}, denoted by $\SP^\bullet_{R,\ur}$,  as:
$$ \SP^\bullet_{R,\ur} := R \otimes_{\U(L)} \SP^\bullet_L.
$$
It computes the total derived tensor product $R \Lotimes_{\U(L)} \dC[\ux]$, and its $0$-cohomolgy is
$$R \otimes_{\U(L)} \dC[\ux] = \frac{R}{R \langle r_1,\dots,r_m \rangle}.$$

Now we will study some conditions on $\ur = (r_1,\dots,r_m)$ implying that the Spencer complex  $\SP^\bullet_{R,\ur}$ is a $R$-free resolution of $R/R \langle r_1,\dots,r_m \rangle$.

\begin{proposition} \label{prop:SpencerResolution}
Under the above hypotheses,
assume that the sequence of symbols $\{\sigma_1(r_i)\ |\ i=1,\dots,m\}$ (resp.
$\{\sigma_0(r_1),\sigma_1(r_2),\dots,\sigma_1(r_m)\}$)
is regular in
 $\Gr R $. Then, the Spencer complex $\SP^\bullet_{R,\ur}$ is concentrated in degree 0 and so it is a (left) $R$-free resolution of $R/R\langle r_1,\dots,r_m \rangle$.
Moreover, $\{r_1,\dots,r_m\}$ is an involutive basis of the ideal $R\langle r_1,\dots,r_m \rangle$, i.e. their symbols
  generate the ideal $\sigma\left( R\langle r_1,\dots,r_m \rangle \right)$.
\end{proposition}

\begin{proof} Let us prove the proposition in the case where $r_1\in F_0 R$, $r_i \in F_1 R$ for $i=2,\dots, m$ and where $\{\sigma_0(r_1),\sigma_1(r_2),\dots,\sigma_1(r_m)\}$ is a regular sequence in
$\Gr R $.
\medskip

We have $L = L_0 \oplus L_1$ with
$L_0= \left( \dC[\ux]\cdot r_1 \right)$ and $L_1=\bigoplus_{i=2}^m \left( \dC[\ux]\cdot r_i\right)$.
Observe that $L_0$ is an ideal of the Lie-Rinehart algebra $L$.
\medskip

As in the proof of Proposition \ref{prop:res-Sp} and \cite[Th. 5.9]{calde_nar_compo}, we are going to filter the complex $\SP^\bullet_{R,\ur}$ in such a way that the graded complex coincides with the Koszul complex of
$$\sigma_0(r_1),\sigma_1(r_2),\dots,\sigma_1(r_m) \in \Gr R  =\Sym_{\dC[\ux]} \Gr_1 R.$$
Instead of declaring $\bigwedge^e L$ to be of order $e$, we now have to use the decomposition $L = L_0 \oplus L_1$, with $L_0$ of order $0$ and $L_1$ of order $1$. Namely, we consider the grading
$ \bigwedge^e L = \left(\bigwedge^e L\right)_{e-1} \bigoplus \left(\bigwedge^e L\right)_e$ with
$$\textstyle \left(\bigwedge^e L\right)_{e-1} =
L_0 \otimes_{\dC[\ux]} \left( \bigwedge^{e-1} L_1 \right),\quad
\left(\bigwedge^e L\right)_e =
 \bigwedge^e L_1,
$$
and the filtration
\begin{eqnarray*}
&
  F_i \SP^{-e}_{R,\ur} = F_i \left( R \otimes_{\dC[\ux]}  \bigwedge^e L \right) :=
\left[\left( F_{i-e+1} R \right) \otimes_{\dC[\ux]} \left(\bigwedge^e L\right)_{e-1} \right] \bigoplus
\left[\left( F_{i-e} R \right) \otimes_{\dC[\ux]} \left(\bigwedge^e L\right)_e \right],&
\end{eqnarray*}
which is easily seen to be compatible with the differentials. The corresponding graded complex is isomorphic to the Koszul complex over $\Gr R \simeq \Sym \Gr_1 R$ associated with the $\dC[\ux]$-linear map
$$ \sigma(L) := \sigma_0(L_0) \oplus \sigma_1(L_1) := \left( \dC[\ux]\cdot \sigma_0(r_1) \right) \oplus \left(\oplus_{i=2}^m \left( \dC[\ux]\cdot \sigma_1(r_i)\right) \right) \hookrightarrow \Gr R,
$$
i.e. the Koszul complex  over $\Gr R$ associated with the sequence $\{\sigma_0(r_1),\sigma_1(r_2),\dots,\sigma_1(r_m)\}$, which by the hypotheses is acyclic in degree $\neq 0$, and we conclude that
$\SP^\bullet_{R,\ur}$ is also acyclic in degree $\neq 0$.
\medskip

To prove the involutivity of $\{r_1,\dots,r_m\}$ one proceeds as in \cite[Prop. 4.1.2]{calde_ens}.
\medskip

The case where $r_i \in F_1 R$ for $i=1,\dots, m$ and $\{\sigma_1(r_1),\sigma_1(r_2),\dots,\sigma_1(r_n)\}$ is a regular sequence in $\Gr R $ is easier and can be proven in a similar way by considering the filtration
\begin{eqnarray*}
&
F_i \SP^{-e}_{R,\ur} :=
\left( F_{i-e} R \right) \otimes_{\dC[\ux]} \bigwedge^e L,&
\end{eqnarray*}
and checking that the corresponding graded complex is isomorphic the Koszul complex  over $\Gr R$ associated with the sequence $\{\sigma_1(r_1),\sigma_1(r_2),\dots,\sigma_1(r_m)\}$.
\end{proof}

\section{Free divisors, the strong Koszul hypotheses and the Bernstein module}
\label{sec:FreeDiv}

From now on, we will write $V=\dC^n$ and
$\widetilde{V}=\dC\times V$. We let
$(w_1,\ldots,w_n)$ be coordinates on $V$, and
$(w_0,w_1,\ldots,w_n)$ coordinates on $\widetilde{V}$. We will write
$$
 A_V:=\dC[w_1,\dots,w_n],\  D_V =  A_V\langle \partial_{w_1},\dots,\partial_{w_n}\rangle.
$$
 We assume that $h\in A_V$ is a reduced quasi-homogeneous polynomial with weights $(p_1,\dots,p_n)$ of degree $d$ and that $D=\{h=0\}\subset \dA^n$ is a free divisor in the sense of \cite{KS1}, that is, that the module $\Der_V(-\log D)$ is free over $\cO_V$. Let $\delta_1,\dots,\delta_{n-1},\delta_n=\chi=\sum_{i=1}^n p_i w_i \partial_{w_i}$ be a basis of $\Der(-\log D) \subset \Der_\dC( A_V)$, chosen in such a way that $\delta_i(h)=0$ for $i=1,\dots,n-1$.
\medskip

Notice that the ring of logarithmic differential operators $A_V[\delta_1,\dots,\delta_n] \subset D_V$
is actually equal to the enveloping algebra $\U(\Der(-\log D))$ \cite[Prop. 2.2.5]{calde_ens}.
We assume for the moment the following {\em strongly Koszul} hypothesis
(\cite[Def. 7.1]{granger_schulze_rims_2010}, \cite[Cor. (1.12)]{nar_symmetry_BS}):
\begin{itemize}
\item[(SK)] The symbols with respect to the usual order filtration in $ D_V$ of $h,\delta_1,\dots,\delta_{n-1}$ form a regular sequence in $\Gr  D_V$, or equivalently, the symbols with respect to the total order filtration in $ D_V[s]$ of $h,\delta_1,\dots,\delta_{n-1},\chi-ds$ form a regular sequence in $\Gr  D_V[s]$.
\end{itemize}

Hypothesis (SK) makes sense not only for polynomial quasi-homogeneous free divisors as above, but also for free divisors on any complex manifold.
Examples of free divisors satisfying (SK) are those which are locally quasi-homogeneous (\cite[Th. 5.9]{calde_nar_compo}), for instance normal crossing divisors, free hyperplane arrangements, or the
discriminant of stable maps in Mather's ``nice dimensions''. Later we will be concerned with the more special class of so-called \emph{linear free divisors} (see Definition \ref{def:LFD} below). These are discriminants in prehomogenous vector spaces, and then the (SK) condition is equivalent to a finite orbit type assumption for a natural group action.
\medskip

Hypothesis (SK) implies the following properties (\cite[Criterion 4.1]{cas_ucha_exper},
\cite[Th. 1.24]{calde_nar_lct_ilc}, \cite[\S 4]{nar_symmetry_BS}):

\begin{enumerate}
\item[(a)] The natural map
$$ D_V[s] \Lotimes_{A_V[\delta_1,\dots,\delta_n][s]}   A_V[s]h^s \to  D_V[s] h^s
$$ is an isomorphism, or equivalently:
\begin{enumerate}
\item[(a-1)] The annihilator of $h^s$ over $ D_V[s]$ is generated by $\delta_1,\dots,\delta_{n-1},\chi-ds$; and
\item[(a-2)] The Spencer complex over $ D_V[s]$ associated with $(\delta_1,\dots,\delta_{n-1},\chi-ds)$ is exact in degrees $\neq 0$, i.e.  it is a $ D_V[s]$-free resolution of
$$ D_V[s]/ D_V[s]\langle \delta_1,\dots,\delta_{n-1},\chi-ds \rangle.$$
\end{enumerate}
\item[(b)] The natural map
$$ D_V[s] \Lotimes_{A_V[\delta_1,\dots,\delta_n][s]}  \frac{ A_V[s]h^s}{ A_V[s] h^{s+1}} \to \frac{ D_V[s] h^s}{ D_V[s] h^{s+1}}
$$
is an isomorphism, or equivalently:
 \begin{enumerate}
\item[(b-1)] The annihilator of the class of $h^s$ over $ D_V[s]$ is generated by $h,\delta_1,\dots,\delta_{n-1},\chi-ds$; and
\item[(b-2)]The Spencer complex over $ D_V[s]$ associated with $(h,\delta_1,\dots,\delta_{n-1},\chi-ds)$  is exact in degrees $\neq 0$, i.e. it is a $ D_V[s]$-free resolution of
$$ D_V[s]/ D_V[s]\langle h,\delta_1,\dots,\delta_{n-1},\chi-ds \rangle.$$
\end{enumerate}
This property implies that the $b$-function $b_h(s)$ of $h$ satisfies the symmetry: $b_h(-s-2)=\pm b_h(s)$.
\item[(c)] The Logarithmic Comparison Theorem holds, or equivalently in terms of $ D_V$-module theory, the natural map
$$  D_V \Lotimes_{A_V[\delta_1,\dots,\delta_n]}   A_V(D) \to  A_V[\star D]
$$
is an isomorphism. This property is equivalent to the following facts:
\begin{enumerate}
\item[(c-1)] The $ D_V$-module of meromorphic functions $ A_V[\star D]$ is generated by $h^{-1}$ (this is a consequence of the fact that $b_h(s)$ has no integer roots $< -1$); and
\item[(c-2)]  The Spencer complex over $ D_V$ associated with $(\delta_1,\dots,\delta_{n-1},\chi+d)$ is exact in degrees $\neq 0$.
\end{enumerate}
\end{enumerate}

As a consequence of (b-1), the $b$-function $b_h(s)$ belongs to $ D_V[s]\langle h,\delta_1,\dots,\delta_{n-1},\chi-ds \rangle$. Actually, it is the generator of $\dC[s]\cap D_V[s]\langle h,\delta_1,\dots,\delta_{n-1},\chi-ds \rangle$.
\medskip

Let us consider now a new variable $w_0$ and the rings
$$\tOO= A_V[w_0]=\dC[w_0,w_1,\dots,w_n],\quad \tDD= D_V[w_0]\langle\partial_{w_0}\rangle=\tOO\langle\partial_{w_0},\partial_{w_1},\dots,\partial_{w_n}]\rangle.
$$
Let us consider $\tth = h-cw_0^d$, $\tchi = \chi + w_0 \partial_{w_0}$ with $c\in \dC\setminus\{0\}$. We are interested in the ideals $I(\beta)=\tDD\langle \tth, \delta_1,\dots,\delta_{n-1},\tchi-d\beta\rangle$ for some complex parameter $\beta$ and the $\tDD$-module $N(\beta)=\tDD/I(\beta)$.
\medskip

Let also consider the ring $\tDD[s]$ endowed with the total order filtration, the ideal
$$ I(s) = \tDD[s]\langle \tth, \delta_1,\dots,\delta_{n-1},\tchi-ds\rangle \subset \tDD[s].
$$
and the $ D_V[s]$-module $N(s)=\tDD[s]/I(s)$.

Let us denote by $\widetilde{\SP^\bullet}(\beta)$ the Spencer complex over $\tDD$ associated with $(\tth, \delta_1,\dots,\delta_{n-1},\tchi-d\beta)$, and let $\widetilde{\SP^\bullet}(s)$ be the Spencer complex over $\tDD[s]$ associated with $(\tth, \delta_1,\dots,\delta_{n-1},\tchi-ds)$.

\begin{proposition} \label{prop:free-res-cM(s)}
The complex $\widetilde{\SP^\bullet}(s)$ is concentrated in degree $0$ and so it is a free resolution of $N(s)$. Moreover $\tth, \delta_1,\dots,\delta_{n-1},\tchi-ds$ is an involutive basis of $I(s)$.
\end{proposition}

\begin{proof}
We are going to use Proposition \ref{prop:SpencerResolution} for the case $R=\tDD[s]$ together with the
total order filtration (for which $\partial_{w_0}, \ldots,\partial_{w_n}$ as well as $s$ have degree $1$).
Notice that the symbols of the generators of $I(s)$ with respect to that filtration are:
$$ \tth = h - c w_0^d, \sigma(\delta_1),\dots,\sigma(\delta_{n-1}),\sigma(\chi)-ds+w_0\xi_0.
$$
We have to show that they form a regular sequence in $\Gr \tDD[s] =\tOO[s,\xi_0,\dots,\xi_n]$.
We already know by the (SK) assumption that $h, \sigma(\delta_1),\dots,\sigma(\delta_{n-1}),\sigma(\chi)-ds$ is a regular sequence in $ A_V[s,\xi_1,\dots,\xi_n]$.

To show that $ \tth = h - c w_0^d, \sigma(\delta_1),\dots,\sigma(\delta_{n-1}),\sigma(\chi)-ds+w_0\xi_0$ is a regular sequence
in $\tOO[s,\xi_0,\ldots,\xi_n]$, we first notice that
$h, \sigma(\delta_1),\dots,\sigma(\delta_{n-1}),\sigma(\chi)-ds,w_0$
 is a regular sequence in $ A_V[s,\xi_1,\dots,\xi_n][w_0]=\tOO[s,\xi_1,\dots,\xi_n]$. Since
$$ \langle h, \sigma(\delta_1),\dots,\sigma(\delta_{n-1}),\sigma(\chi)-ds,w_0\rangle = \langle h - c w_0^d, \sigma(\delta_1),\dots,\sigma(\delta_{n-1}),\sigma(\chi)-ds,w_0\rangle
$$
we deduce that $\tth, \sigma(\delta_1),\dots,\sigma(\delta_{n-1}),\sigma(\chi)-ds,w_0$ is a regular sequence.
On the other hand,
$$\tth, \sigma(\delta_1),\dots,\sigma(\delta_{n-1}),\sigma(\chi)-ds,w_0,\xi_0$$
is again a regular sequence in $\tOO[s,\xi_1,\dots,\xi_n][\xi_0]$, and in a similar way we deduce
that
$$ \tth, \sigma(\delta_1),\dots,\sigma(\delta_{n-1}),\sigma(\chi)-ds+w_0\xi_0,w_0,\xi_0
$$
is a regular sequence. We conclude that
$ \tth, \sigma(\delta_1),\dots,\sigma(\delta_{n-1}),\sigma(\chi)-ds+w_0\xi_0
$
is a regular sequence.
\end{proof}

\begin{proposition} For any $\beta \in \dC$, the multiplication
$ (s-\beta): N(s) \to N(s)$
is injective.
\end{proposition}

\begin{proof}
Since the generators of $I(s)$ form an involutive basis and $\sigma(s-\beta)=s$, it is enough to check that the following sequence
$$s,\tth, \sigma(\delta_1),\dots,\sigma(\delta_{n-1}),\sigma(\chi)-ds+w_0\xi_0$$
is regular in $\Gr \tDD[s] = \tOO[s,\xi_0,\dots,\xi_n]$ .
\medskip

We know that $\sigma(\delta_1),\dots,\sigma(\delta_{n-1}),\sigma(\chi)$ is a regular sequence in $\Gr  D_V = A_V[\xi_1,\dots,\xi_n]$  (this is the Koszul property). So, $\sigma(\delta_1),\dots,\sigma(\delta_{n-1}),\sigma(\chi),s$ is a regular sequence in $\Gr  D_V[s] = A_V[\xi_1,\dots,\xi_n,s]$.
\medskip

Let us prove that
$h-c w_0^d,\sigma(\delta_1),\dots,\sigma(\delta_{n-1}),\sigma(\chi),s$ is a regular sequence in $ A_V[w_0][\xi_1,\dots,\xi_n,s]$. We filter by the degree in $w_0$ and since
$w_0^d,\sigma(\delta_1),\dots,\sigma(\delta_{n-1}),\sigma(\chi),s$ is a regular sequence, we are done.
Now, we add the new variable $\xi_0$ and we know that
$$ h - c w_0^d,\sigma(\delta_1),\dots,\sigma(\delta_{n-1}),\sigma(\chi),s,\xi_0
$$
is a regular sequence in $ A_V[w_0][\xi_0,\xi_1,\dots,\xi_n,s]$.
We repeat the procedure in the proof of Proposition \ref{prop:free-res-cM(s)} and we deduce first that
$$ h - c w_0^d,\sigma(\delta_1),\dots,\sigma(\delta_{n-1}),\sigma(\chi)+w_0 \xi_0-ds,s,\xi_0
$$
is a regular sequence, and second that
$$ h - c w_0^d,\sigma(\delta_1),\dots,\sigma(\delta_{n-1}),\sigma(\chi)+w_0 \xi_0-ds,s
$$
is a regular sequence.
\end{proof}

\begin{corollary} \label{cor:free-res-cM(beta)}
For any $\beta \in \dC$, the Spencer complex over $\tDD$ associated with
$(\tth, \delta_1,\dots,\delta_{n-1},\tchi-d\beta)$ is a free resolution of $N(\beta)$.
\end{corollary}

\begin{proof} We proceed as in the proof of \cite[Cor. (4.5)]{nar_symmetry_BS}:
\begin{eqnarray*}
& \displaystyle N(\beta) = \frac{\tDD[s]}{\tDD[s]\langle s-\beta\rangle} \otimes_{\tDD[s]} N(s) = \frac{\tDD[s]}{\tDD[s]\langle s-\beta\rangle} \Lotimes_{\tDD[s]} N(s) = &\\
& \displaystyle \frac{\tDD[s]}{\tDD[s]\langle s-\beta\rangle} \otimes_{\tDD[s]} \widetilde{\SP^\bullet}(s) = \widetilde{\SP^\bullet}(\beta).
\end{eqnarray*}
\end{proof}

\begin{proposition}
For any complex parameter $\beta\in\dC$, the $\tDD$-module $N(\beta)$ is holonomic and
the generators
$$h - c w_0^d,\delta_1,\dots,\delta_{n-1},\tchi-d\beta$$ form an involutive basis of $I(\beta)$.
\end{proposition}

\begin{proof} The proposition is a consequence of the fact that the symbols of these generators
$$ h - c w_0^d,\sigma(\delta_1),\dots,\sigma(\delta_{n-1}),\sigma(\chi)+w_0 \xi_0
$$
form a regular sequence in $\Gr \tDD$, and this is proven following the same lines as in the proofs of the two preceding propositions.
\end{proof}

Now we are concerned with the question of invertibility of the multiplication
$ w_0: N(\beta)  \to N(\beta)$.
After Corollary \ref{cor:free-res-cM(beta)}, we are reduced to study the cokernel of the \underline{injective} map
$w_0:\widetilde{\SP^\bullet}(\beta) \to \widetilde{\SP^\bullet}(\beta)$.

\begin{theorem}\label{theo:Invertibility} The cokernel of $w_0:\widetilde{\SP^\bullet}(\beta) \to \widetilde{\SP^\bullet}(\beta)$ is acyclic whenever the following condition holds:
$$ \beta \notin \bigcup_{k\geq 0} \left( \frac{k}{d} + \{\text{\rm roots of\ }\ b_h(s)\} \right).
$$
\end{theorem}

\begin{proof}
Let us call $\cK^\bullet$ the cokernel of $w_0:\widetilde{\SP^\bullet}(\beta) \to \widetilde{\SP^\bullet}(\beta)$ and $\widetilde{L}(\beta) \subset \tDD$ the Lie-Rinehart algebra over $(\tOO,\dC)$ with basis $h - c w_0^d,\delta_1,\dots,\delta_{n-1},w_0\partial_{w_0}+\chi-d\beta$ .
\medskip

We
consider the filtration $F^{\partial_{w_0}}_\bullet \tDD$ given by the order with respect to $\partial_{w_0}$.
The graded ring is $\Gr^{\partial_{w_0}} \tDD =  D_V[w_0][\xi_0]$.
\medskip

Let us call $\cQ := \tDD/\langle w_0 \rangle \tDD$, that can be naturally identified, as left $ D_V$-module, with $ D_V[\partial_{w_0}]$. From the identity $\partial_{w_0}^j w_0 = w_0 \partial_{w_0}^j + j \partial_{w_0}^{j-1}$ we see that the exact sequence of $( D_V;\tDD)$-bimodules
\begin{equation}  \label{eq:exact-w0}
0 \longrightarrow \tDD \stackrel{w_0 \cdot }{\longrightarrow} \tDD \longrightarrow \cQ =  D_V[\partial_{w_0}] \longrightarrow 0
\end{equation}
is strict with respect to $F^{\partial_{w_0}}_\bullet$ and the right action of $w_0$ on $\cQ =  D_V[\partial_{w_0}]$ is given by
$$ 	\sum_j P_j \partial_{w_0}^j \in  D_V[\partial_{w_0}] \longmapsto \sum_j j P_j \partial_{w_0}^{j-1} \in  D_V[\partial_{w_0}].
$$
So, the right action of $w_0\partial_{w_0}$ on $\cQ =  D_V[\partial_{w_0}]$ is given by
$$ 	\sum_j P_j \partial_{w_0}^j \in  D_V[\partial_{w_0}] \longmapsto \sum_j j P_j \partial_{w_0}^j \in  D_V[\partial_{w_0}].
$$
For each $e=0,\dots,n$ we have $\cK^{-e}= \cQ \otimes_{\tOO} \bigwedge^e \widetilde{L}(\beta) $ and the differentials $d^e:\cK^{-e} \to \cK^{-e}$ are given by the same expression as in (\ref{eq:diff-SP}). Since the right multiplication on $\cQ$ of the elements in $\widetilde{L}(\beta)$ is compatible with the $F^{\partial_{w_0}}_\bullet$-filtration on each $\cK^{-e}$, we may consider the filtration $F^{\partial_{w_0}}_\bullet$ on the whole complex $\cK^\bullet$.
\medskip

Taking the $\Gr^{\partial_{w_0}}$ of (\ref{eq:exact-w0}) we obtain an
 exact sequence of graded $( D_V; D_V[w_0][\xi_0])$-bimodules (here $ D_V$ has the trivial grading)
$$
0 \longrightarrow  D_V[w_0][\xi_0] \stackrel{w_0 \cdot }{\longrightarrow}  D_V[w_0][\xi_0] \longrightarrow \Gr_{\partial_{w_0}}\cQ = D_V[\xi_0] \longrightarrow 0,
$$
where the action of $w_0$ on $\Gr^{\partial_{w_0}}\cQ =  D_V[\xi_0]$ vanishes and the action of
$w_0 \xi_0$ on the degree $k$ piece  $\Gr_k^{\partial_{w_0}}\cQ =  D_V \cdot  \xi_0^k$
 is given by
$$ P \cdot  \xi_0^k \in \Gr^{\partial_{w_0}}_k\cQ =  D_V \cdot  \xi_0^k \longmapsto k P  \cdot  \xi_0^k \in \Gr^{\partial_{w_0}}_k\cQ =  D_V \cdot  \xi_0^k.
$$
So, the degree $k$ piece $\Gr^{\partial_{w_0}}_k\cK^\bullet$ is isomorphic to the Spencer complex $\SP^\bullet_{ D_V,\ur^k}$ over $ D_V$ associated with
$\ur^k =(h,\delta_1,\dots,\delta_{n-1},\chi-d\beta+k)$ and we have
\begin{eqnarray*}
&\displaystyle
 \Gr^{\partial_{w_0}}_k\cK^\bullet \simeq \SP^\bullet_{ D_V,\ur^k} \simeq \left(\frac{ D_V[s]}{ D_V[s]\langle s -(\beta-k/d)\rangle}\right) \otimes_{ D_V[s]} \SP^\bullet_{ D_V[s],\ur^s} \stackrel{\text{(b)}}{\simeq}
&
\\
&\displaystyle
  \left(\frac{ D_V[s]}{ D_V[s]\langle s -(\beta-k/d)\rangle}\right) \Lotimes_{ D_V[s]} \left(\frac{ D_V[s] h^s}{ D_V[s] h^{s+1}}\right),
\end{eqnarray*}
with $\ur^s = (h,\delta_1,\dots,\delta_{n-1},\chi-ds)$. If $b_h(\beta-k/d)\neq 0$, then $s-(\beta-k/d)$ and $b_h(s)$ are coprime and the map
$$ s-(\beta-k/d): \frac{ D_V[s] h^s}{ D_V[s] h^{s+1}} \longrightarrow \frac{ D_V[s] h^s}{ D_V[s] h^{s+1}}
$$
is invertible, and so $ \Gr^{\partial_{w_0}}_k\cK^\bullet$ is acyclic.
\end{proof}

\textbf{Remark:} Actually, we do not need to assume that $h$ is quasi-homogeneous. At most we need to have an Euler vector field, let us say with $\chi(h)=h$. This is actually implied by the (SK) hypothesis (see \cite[Prop. (1.9) and (1.11)]{nar_symmetry_BS}). On the other hand, instead of considering the deformation $\tth = h - c w_0^d$, with $d$ equal to the degree of $h$, we can consider any deformation $\tth = h - c w_0^d$ with $d\geq 1$ arbitrary, including the case $d=1$, and the deformation of $\chi$, assuming $\chi(h)=h$, would be $\tchi = \chi + \frac{1}{d} w_0 \partial_{w_0}$. That covers the case of studying the graph embedding $h-w_0$.
\medskip

Let us also notice that if instead of taking a basis  $\delta_1,\dots,\delta_{n-1},\chi$ as before, with $\delta_i(h)=0$ for $i=1,\dots, n-1$ and $\chi(h)=h$, we take a general basis $\delta_1,\dots,\delta_n$ with $\delta_i(h)=\alpha_i h$ for $i=1,\dots, n$, our deformation ideal would be defined as
$$ I(s) = \tDD[s]\langle \tth = h - c w_0^d, \tdelta_1,\dots,\tdelta_n\rangle$$
with $\tdelta_i = \delta+\frac{\alpha_i}{d} w_0\partial_{w_0}-\frac{\alpha_i}{d}s$.
Observe that $I(s)$ is always contained in the $\tDD[s]$-annihilator of the class of $\tth^s$ in
$\frac{\tDD[s] \tth^s}{\tDD[s] \tth^{s+1}}$.
\medskip

Finally, everything works at the level of germs of analytic functions instead of the global polynomial case.

\section{Tautological systems and Fourier transformation}

We introduce here the main playing character of this paper, which is a certain generalization of
the $A$-hypergemetric system of Gelfand, Kapranov, Graev and Zelevinski (see, e.g.,  \cite{GKZ1}, \cite{Adolphson}).
The main point is that the GKZ-systems are build from a given torus action on an affine space, and
this will be replaced by an action of a more general algebraic group. The $\cD$-module thus obtained
has been considered rather recently in a series of papers by Yau and others (see \cite{Taut1, Taut2, Taut3}), but the idea dates back to \cite{KapranovReduct}
and \cite{HottaEq}.

Let us start with the definition of a tautological system, which we adapt slightly to fit to our purpose.
Recall that we  write $V=\dC^n$, with coordinates $w_1,\ldots,w_n$ and $\widetilde{V}=\dC_{w_0}\times V$.
We denote by $V^\vee$ resp. $\widetilde{V}^\vee$ the dual spaces, with dual coordinates
$(\lambda_1,\ldots,\lambda_n)$ resp. $(\lambda_0,\lambda_1,\ldots,\lambda_n)$.

\begin{definition}\label{def:TautSystem}
Let $G$ be a reductive algebraic group acting on $V$ via $\rho:G \hookrightarrow \Gl(V)$ and let $d\rho:\fg\rightarrow\text{End}(V)$ be the associated Lie algebra
action. For any $x\in\fg$, we write $Z_{(d\rho)(x)}\in\textup{Der}_V$ for the linear vector field on $V$ given by
$$
Z_{(d\rho)(x)}(g)(w):=\frac{d}{dt}g(\rho(e^{-tx})(w))_{|t=0}.
$$
Let moreover $X\subset V$ be a closed subvariety of $V$ which is $G$-invariant, i.e., a union of $G$-orbits. Chose a Lie algebra homomorphism $\beta:\fg\rightarrow \dC$. Then we
consider the left ideal
$$
\cI(G,\rho,X,\beta):=\cD_V(I(X))+\cD_V(Z_{(d\rho)(x)}-\beta(x))_{x\in \fg} \subset \cD_V
$$
and the quotient $\check{\cM}(G,\rho,X,\beta)=\cD_V/\cI(G,\rho,X,\beta)$.
Moreover, we put
$$
\cM(G,\rho,X,\beta):=\FL(\check{\cM}(G,\rho,X,\beta))\in\text{Mod}(\cD_{V^\vee})
$$
and call $\cM(G,\rho,X,\beta)$ the \emph{tautological system} associated to $G$, $\rho$, $X$ and $\beta$.
If all the input data are clear from the context, we also write
$\check{\cM}:=\check{\cM}(G,\rho,X,\beta)$ and $\cM:=\cM(G,\rho,X,\beta)$.
\end{definition}

Below we will consider, for a given tuple $(G,\rho, V, \beta)$, a homogenized version of the action $\rho$, namely, we let
$\widetilde{G}:=\dG_m\times G$ and we consider $\widetilde{V}:=\dA^1\times V$ together with the extended action
$$
\begin{array}{rcl}
  \widetilde{\rho}:\widetilde{G} & \longrightarrow & \text{Aut}(\widetilde{V}) \\
  (t,g) & \longmapsto & \left[(x_0,\underline{x})\mapsto (tx_0,t\rho(g)(x) \right].
\end{array}
$$
Given a $G$-variety $X\subset V$, let $\widetilde{X}$ be the closure of its cone in $\widetilde{V}$
$$
\widetilde{X}:=\overline{\left\{(t,tx)\in\widetilde{V}\,|\,t\in\dG_m, x\in X\right\}.}
$$
We will consider the ``extended'' differential systems
$$
\check{\cM}(\widetilde{G},\widetilde{\rho},\widetilde{X},\widetilde{\beta})\in\text{Mod}(\cD_{\widetilde{V}})
\quad\quad\text{resp.}\quad\quad
\cM(\widetilde{G},\widetilde{\rho},\widetilde{X},\widetilde{\beta})\in\text{Mod}(\cD_{\widetilde{V}^\vee}),
$$
where we write $\widetilde{\beta}:\widetilde{\fg}\cong\dC\times\fg\rightarrow\dC$ for any Lie algebra homomorphism restricting
to $\beta$ on $\fg$.\\

We are going to apply the above construction in the setup where the group and its action is defined by
what is called a \emph{linear free divisor} (see \cite{BM}). Let us recall the basic notion.

\begin{definition}\label{def:LFD}
Let $D\subset V$ be a reduced divisor. Suppose that it is free, i.e., that $\Der(-\log D)$ is $\cO_V$-free.
Then $D$ is called linear free if there is a basis $\delta_1,\ldots,\delta_n$ of $\Der(-\log D)$ such
that we have $\delta_i=\sum_{j=1}^{n} a_{ji}\partial_{w_j}$ where $a_{ji}\in \dC[w_1,\ldots,w_n]_1$ are linear forms.
\end{definition}

Let $D\subset V$ be a linear free divisor and write $h\in\dC[w_1,\ldots,w_n]$ for its defining equation,
then $h$ is a homogeneous polynomial of degree $n$
since the matrix $S:=(a_{ij})_{i,j=1,\ldots,n}$ (called Saito matrix) has the property that $\det(S)=h$ (see \cite[Lemma 1.9]{KS1}).

Recall (see, e.g., \cite{GMNS}) that $G_D$ denotes the identity component of $\{g\in\Gl(V)\,|g(D)=D\}$. We call the linear free divisor $D$ reductive if $G_D$ is so. A major class of examples of linear free divisors come from quiver representations, they are all reductive. However, there are non-reductive linear free divisors, see, e.g., the example after
\cite[Definition 2.1]{GMNS}. In the sequel of this paper, we will only be concerned with the reductive case.

The Lie algebra $\fg_D$ of $G_D$ acts on $V$ via derivations, and we have the Lie algebra isomorphism
$$
\begin{array}{rcl}
  \fg_D & \longrightarrow & \text{Der}(-\log D)_0 \\
  A & \longmapsto & \underline{w}\cdot A^{tr}\cdot \underline{\partial}_{\underline{w}}.
\end{array}
$$
Here $\text{Der}(-\log D)_0$ is the set of logarithmic derivations along $D$ of degree $0$ (notice that since
$D$ is linear free, the module $\Der(-\log D)$ inherits the natural grading of $\Der_V$, where the variables $w_i$ have degree $1$ and partial derivatives $\partial_{w_i}$ have degree $-1$).
Similarly, we let $A_D$ be the unity component of the group $\{g\in\Gl(V)\,|\,g^*h=h\}$. We have
$\fg_D=\fa_D\oplus \dC \cdot \chi$, where $\chi=\sum_{i=1}^{n} w_i\partial_{w_i}$ (this vector field was also called $\delta_n$ in section \ref{sec:FreeDiv}, where
it was defined for any quasi-homogeneous free divisor).
Notice that the pair $(V,G_D)$ is a prehomogeneous vector space (see, e.g., \cite{Kim}), with
discriminant locus $D$ and open orbit $V\backslash D$.
Let us also recall that a linear free divisor $D\subset V$ satisfies the (SK) condition if and only if the stratification
of $D$ by orbits of $A_D$ is finite \cite[Prop. 7.2]{granger_schulze_rims_2010}.

We are going to study the tautological system as well as its extended version
for the group $G:=A_D$. Let $\rho:A_D\rightarrow \Gl(V)$ denotes the action of $A_D$ on $V$.
Moreover, chose a point $p\in V\backslash D$ and put $X:=\overline{\rho(A_D)(p)}$.
Actually, our construction (in particular, the tautological system
associated to the divisor $D$) does not depend on the choice of the point $p$ up to isomorphism,
but we will not elaborate on this point here.

We have the following lemma, which describes the geometry of the orbit closure $X$.
\begin{lemma}\label{lem:OrbitClosed}
Let as above $D$ a reductive linear free divisor and consider the action $\rho:A_D\rightarrow \Gl(V)$ and  its extended version $\widetilde{\rho}:\widetilde{A_D}\rightarrow\Gl(\widetilde{V})$ (recall that $\widetilde{A_D}=\dG_m\times A_D$). Then we have the following facts.
\begin{enumerate}
  \item The orbit $\rho(A_D)(p)$ is closed, i.e., we have $X=\rho(A_D)(p)$.
  \item Consider the extended action
  $\widetilde{\rho}:\widetilde{A_D}\rightarrow\Gl(\widetilde{V})$ (recall that $\widetilde{A_D}=\dG_m\times A_D$). Put $\widetilde{p}:=(1,p)$ and $\widetilde{X}:=\overline{\widetilde{\rho}(\widetilde{A_D}(\widetilde{p}))}$
  then
  $$
  \widetilde{X} \backslash \widetilde{\rho}(\widetilde{A_D})(\widetilde{p}) \subset \{0\}\times V\subset \widetilde{V}.
  $$
\end{enumerate}
\end{lemma}

\begin{proof}
\begin{enumerate}
  \item As has been shown in \cite[Section 3]{dGMS}, the orbit $\rho(A_D)(p)$ is nothing but the fibre
  $h^{-1}(h(p))$, which is obviously a closed subvariety of $V$.
  \item This follows directly from the definition of the action $\widetilde{\rho}$ and from part 1.: By definition,
  the orbit $\widetilde{\rho}(\widetilde{A_D}) \subset \dG_m\times V$ is simply the cone over the orbit
  $\rho(A_D)\subset V$, hence closed in $\dG_m\times V$ by the first point. Hence the boundary of its
  closure in $\widetilde{V}$ is contained in the divisor $\{0\}\times V$.
\end{enumerate}
\end{proof}

The next step is to give a more explicit description for the extended system $\check{\cM}(\widetilde{G},\widetilde{\rho},\widetilde{X},\widetilde{\beta})$ for the case $G=A_D$.
We consider the dual action $\rho^\vee:G=A_D\rightarrow \Gl(V^\vee)$. As has
been shown in \cite[Proposition 3.7]{dGMS}, since $G$
is reductive, this action is again prehomogeneous, with discriminant locus (i.e., the complement of the open orbit) a divisor, which we call dual divisor of $D$ and which we denote by $D^\vee\subset V^\vee$.
\begin{lemma}\label{lem:CalculationTautSystem}
Let $D\subset V$ be
a reductive linear free divisor, and let $A_D$, $\rho$, $X$ be as above.
Put $\beta:=0$ and $\widetilde{\beta}:=(\beta_0,0)$.
Then
\begin{equation}\label{eq:DescriptionFLTaut}
\check{\cM}=\check{\cM}(\widetilde{G},\widetilde{\rho},\widetilde{X},\widetilde{\beta})
=\cD_{\widetilde{V}}/(h(p)w_0^n-h,\delta_1,\ldots,\delta_{n-1},\widetilde{\chi}-\beta_0),
\end{equation}
where $\delta_1,\ldots,\delta_{n-1}$ is a basis of $\Der(-\log\,h)$ and where $\widetilde{\chi}=w_0\partial_{w_0}+\sum_{i=1}^{n}
w_n\partial_{w_n}$.

As a consequence, we have
$$
\cM=\cM(\widetilde{G},\widetilde{\rho},\widetilde{X},\widetilde{\beta})
=\cD_{\widetilde{V}^\vee}/(h(p)\partial_{\lambda_0}^n-h(\partial_{\lambda_1},\ldots,\partial_{\lambda_n}),
\delta^\vee_1,\ldots,\delta^\vee_{n-1},\widetilde{\chi}^\vee+(n+1)+\beta_0).
$$
Here $\widetilde{\chi}^\vee=\sum_{i=0}^{n} \lambda_i\partial_{\lambda_i}$ and
$\delta^\vee_1,\ldots,\delta^\vee_{n-1}$ is a basis of $\Der(-\log\,h^\vee)$, where $h^\vee$ is
a reduced equation of the dual divisor $D^\vee\subset V^\vee$.
\end{lemma}

\begin{proof}
We have $I(\widetilde{X})=(h(p)w_0^n-h)$ since $\deg(h)=n$. Moreover, for any $x\in \fa_D$, the linear vector
field $Z_{d\rho(x)}$ is an element in $\Der(-\log\-h)$. On the other hand, we have $\widetilde{\fa_D}=\dC\times \fa_D$,
and for the element $x=(1,0)\in\widetilde{\fa_D}$, the corresponding vector field $Z_{d\rho(x)}$ is nothing
but $\widetilde{\chi}$.
Hence we get $\check{\cM}=\cD_{\widetilde{V}}/(h(p)w_0^n-h,\delta_1,\ldots,\delta_{n-1},\widetilde{\chi}-\beta_0),
$ according to the definition of $\check{\cM}$.

To show the second statement, remark that under the isomorphism of $\dC$-algebras
$$
\begin{array}{rcl}
\Gamma(\widetilde{V},\cD_{\widetilde{V}})=\dC[w_0,\ldots,w_n]\langle \partial_{w_0},\ldots,\partial_{w_n}\rangle &
\longrightarrow &
\dC[\lambda_0,\ldots,\lambda_n]\langle \partial_{\lambda_0},\ldots,\partial_{\lambda_n}\rangle =\Gamma(\widetilde{V}^\vee,\cD_{\widetilde{V}^\vee})\\ \\
w_i & \longmapsto & \partial_{\lambda_i} \\
\partial_{w_i} & \longmapsto & -\lambda_i
\end{array}
$$
corresponding to the Fourier-Laplace transformation functor, we have
$$
\widetilde{\chi}-\beta_0=-\sum_{i=0}^{n} \partial_{\lambda_i}\lambda_i -\beta_0
=-\left(\sum_{i=0}^{n}\lambda_i\partial_{\lambda_i}+(n+1)+\beta_0\right).
$$
Moreover, the dual divisor $D^\vee\subset V^\vee$ is free since $G_D$ is reductive (see \cite[Proposition 3.7]{dGMS}),
and the module $\Der(-\log\,h^\vee)$ is generated by the image of $\fa_D$ under the morphism
$$
\begin{array}{rcl}
\fg_D & \longrightarrow & \Der(-\log D^\vee)_0 \\ \\
  A & \longmapsto & -\underline{\lambda}\cdot A \cdot \underline{\partial}.
\end{array}
$$
But this implies that a basis element $\delta_i$ of $\Der(-\log\,h)$ is sent under the Fourier-Laplace isomorphism
to an basis element $\delta_i^\vee$ of $\Der(-\log\,h^\vee)$.
\end{proof}

The next step is to obtain a more functorial description of both $\cM(\widetilde{G},\widetilde{\rho},\widetilde{X},\widetilde{\beta})$ and $\check{\cM}(\widetilde{G},\widetilde{\rho},\widetilde{X},\widetilde{\beta})$.
This has been carried out for the case $G=\dG^m$ in \cite{SchulWalth2} and used extensively in \cite{Reich2, ReiSe2}.

Let $\widetilde{X}^0$ be the ``open part'' of $\widetilde{X}$, i.e., $\widetilde{X}^0:=
\widetilde{\rho}(\widetilde{A_D})(1,p) \subset \widetilde{X}$. Write $k:\widetilde{X}^0\hookrightarrow \widetilde{V}$
for the composition of the closed embedding $k':\widetilde{X}^0\hookrightarrow \dG_m\times V$ (see
the second point of the Lemma \ref{lem:OrbitClosed}) with
the canonical open embedding $j:\dG_m\times V \hookrightarrow \widetilde{V}$.
Notice that we have an isomorphism
$$
\begin{array}{rcl}
  \iota:\dG_m\times X & \longrightarrow & \widetilde{X}^0 \\
  (t,x) & \longmapsto & (t,tx).
\end{array}
$$
As a matter of notation, for any complex number $\beta_0$,
we write $\cO_{\dG_m}^{\beta_0}:=\cD_{\dG_m}/(t\partial_t-\beta_0)$.
However, from now on we will only consider the case where $\beta_0$ is a real
number.
Consider the $\cD_{\dG_m\times X}$ module
$$
\cN^{\beta_0}:=\cO_{\dG_m}^{\beta_0} \boxtimes \cO_X.
$$
Notice that since $\beta_0\in\dR$, the module $\cN^{\beta_0}$ underlies an element of $\MHM(\dG_m\times X, \dC)$
(the abelian category of complex mixed Hodge modules, see, e.g., \cite[Definition 3.2.1]{DettSa}).
Then we have the following result, which gives a functorial description of $\check{\cM}(\widetilde{G},\widetilde{\rho},\widetilde{X},\widetilde{\beta})$ for the case
$\widetilde{\beta}=(\beta_0,0)$.

\begin{proposition}\label{prop:FLTautIsMHM}
Suppose that $D\subset V$ is linear free and satisfies (SK).
Suppose that $\beta_0$ lies inside the good non-resonant set of Theorem \ref{theo:Invertibility}, that is,
$$
\beta_0\notin \bigcup_{k\geq 0} \left(k+n\cdot \{\textup{roots of }b_h(s)\}\right).
$$
Then the module $\check{\cM}=\check{\cM}(\widetilde{G},\widetilde{\rho},\widetilde{X},(\beta_0,0))$ is obtained
as
$$
\check{\cM}= (k\circ \iota)_+\cN^{\beta_0}
$$
Consequently, $\check{\cM}$ underlies a complex mixed Hodge module on $\widetilde{V}$.
\end{proposition}

\begin{proof}
Recall that $k=j\circ k'$, where $k':\widetilde{X}^0\hookrightarrow \dG_m\times V$ is
closed and where $j:\dG_m\times V \hookrightarrow \widetilde{V}$ is the canonical open embedding. From the closedness of $k'$ we conclude that
\begin{equation}\label{eq:k_iota}
(k'\circ \iota)_+ \cN =
\frac{\cD_{\dG_m\times V}}{\left(I(\im(k')),(\theta)_{\theta\in \Der_V(-X)},\widetilde{\chi}-\beta_0\right)}
\end{equation}
(notice that the direct image of $\cO_X$ under the closed
embedding $X\hookrightarrow V$ is given by $\cD_V/(I(X),(\theta)_{\theta\in \Der_V(-X)})$).

It follows by comparing this expression to formula \eqref{eq:DescriptionFLTaut} that $j^+\check{\cM}=(k'\circ \iota)_+ \cN$.
We now use Theorem \ref{theo:Invertibility}, which tells us
that for our choice of $\beta_0$, the multiplication with $w_0$ is invertible on $\check{\cM}$. Hence we have that $\check{\cM}= j_+j^+ \check{\cM}$,
and hence
 $$
 \check{\cM} = j_+j^+\check{\cM}=j_+(k'\circ \iota)_+ \cN^{\beta_0}=(j \circ k'\circ \iota)_+ \cN^{\beta_0}=(k\circ \iota)_+ \cN^{\beta_0},
 $$
 as required.

The last statement follows since we have a direct image functor (with respect to the morphism $k\circ \iota$)
from $\MHM(\dG_m\times X,\dC)$ to $\MHM(\widetilde{V},\dC)$.
\end{proof}

As a consequence, we obtain the following property of the tautological
system associated to a linear free divisor satisfying the (SK) hypothesis.

\begin{theorem}\label{theo:TautSystMHM}
Let $G=A_D$ as above, where $D\subset V$ is a linear free divisor satisfying the (SK) condition.
Put
\begin{equation}\label{eq:Constant_c}
c:=\min\left(\dZ\cap \bigcup_{k\geq 0} \left(k+n\cdot \{\textup{roots of }b_h(s)\}\right)\right).
\end{equation}
Then for all $\beta_0\in\dZ$ with $\beta_0<c$ the tautological system $\cM(\widetilde{G},\widetilde{\rho},\widetilde{X},(\beta_0,0))$ underlies
an object in $\MHM(\widetilde{V}^\vee)$.
\end{theorem}
Before entering into the proof, we have to relate the Fourier-Laplace transformation entering in the definition
of $\cM$ to the Radon transformation of $\cD_{\dP(\widetilde{V})}$-modules, as has been done in \cite{Reich2}, \cite{ReiSe2}
as well as in \cite{SevCast}. We recall the necessary definitions.

\begin{definition}
Denote by $Z\subset\dP(\widetilde{V})\times \widetilde{V}^\vee$ the universal hyperplane given with equation $\sum_{i=0}^nw_i\lambda_i=0$ and by $U:=(\dP(\widetilde{V})\times \widetilde{V}^\vee)\setminus Z$ its complement. Consider the following diagram
$$
\xymatrix{ && U \ar[drr]^{\pi_2^U} \ar[dll]_{\pi_1^U} \ar@{^(->}[d]^{j_U}&& \\ \dP(\widetilde{V}) && \dP(\widetilde{V})\times \widetilde{V}^\vee \ar[ll]_{\pi_1} \ar[rr]^{\pi_2} && \widetilde{V}^\vee\; , \\ && Z \ar[ull]^{\pi_1^Z} \ar@{^(->}[u]_{i_Z} \ar[rru]_{\pi_2^Z} &&}
$$
The Radon transformations are functors from $D^b_{rh}(\cD_{\dP(\widetilde{V})})$ to $D^b_{rh}(\cD_{\widetilde{V}^\vee})$ given by
\begin{align}
\cR & := \pi^Z_{2,+}\pi_1^{Z,+}\cong\pi_{2,+}i_{Z,+}i_{Z}^+\pi_1^+, \notag \\
\cR^\circ & := \pi^U_{2,+}\pi_1^{U,+}\cong\pi_{2,+}j_{U,+}j^+_U\pi_1^+, \notag \\
\cR^\circ_c & :=\pi^U_{2,\dag}\pi_1^{U,+}\cong\pi_{2,+}j_{U,\dag}j^+_U\pi_1^+, \notag \\
\cR_{cst} & :=\pi_{2,+}\pi_1^+. \notag
\end{align}
\end{definition}

\begin{proof}[Proof of Theorem \ref{theo:TautSystMHM}]
Consider the following diagram,
where the dotted arrows denote functors on $\cD$-modules, not maps.
$$
\begin{tikzcd}
\dG_m\times X \ar[hookrightarrow]{r}{k\circ\iota} \ar{dd}{p_2} & \widetilde{V} \ar[, dotted]{rr}{\FL} & & \widetilde{V}^\vee\\
&&\widetilde{V} \backslash\{0\} \ar[hookrightarrow]{lu}{j_0} \ar{d}{\pi} \\
X \ar[hookrightarrow]{r}{g} & V \ar[hookrightarrow]{r}{j} & \dP(\widetilde{V}). \ar[swap, dotted]{uur}{\cR^\circ_c}
\end{tikzcd}
$$
It can be shown along the lines of \cite[Proposition 2.5, Lemma 2.6, Proposition 2.7]{Reich2} that for any $\beta_0\in \dZ$ we have the following isomorphism
in $D^b_{rh}(\cD_{\widetilde{V}^\vee})$
\begin{equation}\label{eq:FLRadon}
\cR^\circ_c((j\circ g)_+ \cO_X) \cong \FL((k\circ\iota)_+(\cO_{\dG_m}^{\beta_0}\boxtimes\cO_X)).
\end{equation}
In particular, since $\FL$ is exact, it shows that the left hand side is actually an element in $\textup{Mod}(\cD_{\widetilde{V}^\vee})$, i.e., that we have
$\cH^i(\cR^\circ_c((j\circ g)_+ \cO_X))=0$ for $i\neq 0$. Notice also that
for all $\beta_0\in \dZ$, we have an isomorphism $p_2^+\cO_X \cong \cO_{\dG_m}^{\beta_0}\boxtimes\cO_X$.

In particular, since the functors entering in the definition of $\cR_c^\circ$ exist at the level of
mixed Hodge modules, we obtain that the $\cD_{\widetilde{V}^\vee}$-module $\FL((k\circ\iota)_+(\cO_{\dG_m}^{\beta_0}\boxtimes\cO_X))$
underlies an object in $\MHM(\widetilde{V}^\vee)$ (notice that since $\beta_0\in \dZ$, we have
that $\cO_{\dG_m}^{\beta_0}$ is an actual Hodge module, i.e., such that its perverse sheaf is
defined over the rational numbers, and not just an element of $\MHM(\dG_m, \dC)$ as in the case
where $\beta_0$ is an arbitrary real number).

To finish the proof of the theorem, we now use Proposition \ref{prop:FLTautIsMHM}. As we assume that $\beta_0<c$, which implies in particular that $\beta_0\notin \cup_{k\geq 0} \left(k+n\cdot \{\textup{roots of }b_h(s)\}\right)$, we can conclude that
$$
\FL((k\circ\iota)_+(\cO_{\dG_m}^{\beta_0}\boxtimes\cO_X)) \cong
\FL(\check{\cM}(\widetilde{G},\widetilde{\rho},\widetilde{X},(\beta_0,0))) =
\cM(\widetilde{G},\widetilde{\rho},\widetilde{X},(\beta_0,0))
$$
which shows the statement of the theorem.
\end{proof}

\textbf{Remark: } As already stated in the introduction, Theorem \ref{theo:TautSystMHM} should be considered
as an analogue to \cite[Theorem 3.5.]{Reich2}, which treats the case
of GKZ-systems, i.e., where our group $G$ is a $d+1$-dimensional algebraic torus acting on an $n+1$-dimensional
affine space (noticed that \cite[Corollary 3.8]{SchulWalth2} plays a key role
in the proof of this latter result in the same way that Theorem \ref{theo:Invertibility} is needed to show Theorem \ref{theo:TautSystMHM}). In the paper \cite{ReiSe3}, this kind of result is pushed further by not only showing
that certain regular GKZ-systems underly mixed Hodge modules but proving that the associated Hodge filtration
is simply the induced filtration by orders of differential operators (up to a shift). One of the main
ingredients was the calculation of the certain $b$-function (or Bernstein-Sato polynomial) of the generator of the total Fourier-Laplace transform
(corresponding to the module $\check{\cM}$ in our notation) along the coordinate hyperplane $w_0$, which was achieved
using general estimations for such $b$-functions from \cite{ReichSevWalth}. In the present situation,
one would be much interested in a similiar result.

The first interesting example is the so-called $\star_3$-quiver (see \cite[Example 5.3.]{GMNS}), here the underlying graph
is of Dynkin type, and hence the corresponding linear free divisor satisfies the (SK) hypothesis.
A Macaulay2 calculation shows that
the $b$-function of the class of $1$ in the module $\check{\cM}(\widetilde{A_D},\widetilde{\rho},\widetilde{X},\widetilde{\beta})$,
i.e., the polynomial $b(s)$
such that
$$
b(w_0\partial_{w_0})\in V^1\cD_{\widetilde{V}} + \left(h(p)w_0^n-h,\delta_1,\ldots,\delta_{n-1},\widetilde{\chi}+9\right).
$$
has roots $-1,-3,-3,-3,-3,-5$ (notice that $\beta_0=-9$ is the largest integer satisfying the assumptions of Proposition \ref{prop:FLTautIsMHM}). This contrasts  \cite[Corollary 3.9]{ReichSevWalth}, which treats a similar question for  the case of GKZ-systems with normal toric rings, and
where all roots are contained in an interval of length smaller than one. The latter result is crucially used in the proof of \cite[Theorem 3.16]{ReiSe3}. Hence we cannot a priori conclude
that the Hodge filtration on $\check{\cM}(\widetilde{A_D},\widetilde{\rho},\widetilde{X},\widetilde{\beta})$
(and consequently the one on $\cM(\widetilde{A_D},\widetilde{\rho},\widetilde{X},\widetilde{\beta})$) is, up to a shift,
given by the order filtration on $\cD_{\widetilde{V}}$ (resp. the order filtration on $\cD_{\widetilde{V}^\vee}$).
On the other hand, as it has been already noticed in
the last remark of section \ref{sec:FreeDiv},
we can also study the ideal
$(w_0-h,\delta_1,\ldots,\delta_{n-1}, n\cdot \widetilde{\chi}-\beta_0)$ for suitable $\beta_0$.
If the analogue of Theorem \ref{theo:Invertibility}
holds for the quotient by this ideal, then it can be shown that it is nothing but the graph embedding
$i_{h,+} \cO_V(*D) h^\gamma$ (not of the module of meromorphic functions itself, but of the twisted version $\cO_V (*D)h^\gamma$ for some suitable $\gamma$). Notice also that for $\beta_0=0$ the module $\check{M}(\widetilde{A_D},\widetilde{\rho},\widetilde{X},(0,0))$ is then obtained as a pullback under a cyclic cover of such a  direct image under the graph embedding.

The roots of the element $[1]$ of this module
are simply shifts of the roots of $b_h$ itself.
This means that in case where these roots are contained
in an interval of length $<1$ (like in the case of $\star_3$, where they are $-4/3,-1,-1,-1,-1,-2/3$),
we may actually be able to detect the Hodge filtration on the graph embedding module. This is closely related to the general problem of how the Hodge filtration behaves on the module $\cO_V(*D)$,
a question that has raised much attention over the last years in the context of birational geometry, see, e.g.
\cite{PopaMustata, PopaICMTalk}.

\section{Hyperplane sections and Gau\ss-Manin systems}
\label{sec:DimReduc}

In this section we discuss the relation of the tautological system $\cM=\cM(\widetilde{G},\widetilde{\rho},\widetilde{X},(\beta_0,0))$
(where $G=A_D$) to the Gau\ss-Manin system of the universal family of hyperplane sections of a Milnor fibre of $D$. This family is the hypothetical Landau-Ginzburg potential
for a (yet to be found) non-toric A-model. Both the tautological and this Gau\ss-Manin system are regular holonomic
$\cD_{\widetilde{V}^\vee}$-modules (and actually underly, using the results of the last section, objects in $\MHM(\widetilde{V}^\vee)$).
We first show that they are are equal up to smooth $\cD_{\widetilde{V}^\vee}$-modules. In a second step, we consider the dimensional reduction briefly discussed in the introduction. It consists in applying a direct image under a morphism from $\widetilde{V}^\vee$ to
$\dA^2$ given by the identity on the first component and the equation of the dual divisor $D^\vee$ as the second component. We obtain a reduced system that has been intensively studied in \cite{dGMS} using algorithmic methods.

We start with the following statement, which is a direct consequence
of the corresponding results in the toric case, as worked out in details in
\cite{Reich2} and \cite{ReiSe,ReiSe2}. Let $c\in \dZ$ be the constant from
formula \eqref{eq:Constant_c}.
\begin{proposition}\label{prop:4TermSequence}
Let $D\subset V$ be a linear free divisor with defining equation $h$
and suppose that $D$ satisfies the (SK) condition. Let $X=h^{-1}(h(p))$, where $p\in V\backslash D$ is a chosen point.
Let $\textup{can}:V\times V^\vee \rightarrow \dA^1_{\lambda_0}$, $(w,\lambda)\mapsto \sum_{i=1}^{n} w_i\lambda_i$
be the canonical pairing. Consider again the closed embedding $g:X\hookrightarrow V$ from above
(see Lemma \ref{lem:OrbitClosed}) and let $\varphi$ be the composition
$$
\varphi=\left(\textup{can}\circ(g,\id_{V^\vee}),\textup{pr}_2\right): X\times V^\vee \longrightarrow \dA^1_{\lambda_0}\times V^\vee \cong \widetilde{V}^\vee.
$$
Then for all $\beta_0\in\dZ$ with $\beta_0 < c$ there is an exact sequence in
$\textup{Mod}(\cD_{\widetilde{V}^\vee})$
$$
0 \longrightarrow H^{n-2}(X,\dC)\otimes_\dC\cO_{\widetilde{V}^\vee}\longrightarrow
\cH^0 \varphi_+ \cO_{X\times V^\vee} \longrightarrow \cM(\widetilde{G},\widetilde{\rho},\widetilde{X},(\beta_0,0))
\longrightarrow H^{n-1}(X,\dC)\otimes_\dC\cO_{\widetilde{V}^\vee}\longrightarrow 0,
$$
where the left- resp. rightmost term are free $\cO_{\widetilde{V}^\vee}$-modules with the trivial
connection (having $H^{n-2}(X,\dC)$ resp. $H^{n-1}(X,\dC)$ as flat sections).
\end{proposition}
\begin{proof}
From the definition of the various Radon transformation functors and the adjunction triangle for the embeddings
$Z\hookrightarrow \dP(\widetilde{V})\times \dP(\widetilde{V}^\vee)$ and $U\hookrightarrow\dP(\widetilde{V})\times \dP(\widetilde{V}^\vee)$
we obtain exact triangles
$$
\begin{array}{cccccc}
\cR(M)[-1] & \longrightarrow & \cR_{cst}(M) & \longrightarrow & \cR^\circ(M) & \stackrel{+1}{\longrightarrow} \\
\cR^\circ_c(M) & \longrightarrow & \cR_{cst}(M) & \longrightarrow & \cR(M)[+1] & \stackrel{+1}{\longrightarrow}
\end{array}
$$
for any $M\in D^b_{rh}(\cD_{\dP(\widetilde{V})})$ (where the second is dual to the first), see \cite[Proposition 2.4]{Reich2} for details.
Recall (see the discussion after formula \eqref{eq:FLRadon})
that we have $\cH^i\cR^\circ_c((j\circ g)_+\cO_X) =0$ for $i\neq 0$.
Moreover, it can be shown as in \cite[Proposition 2.7]{Reich2} that
$$
\cR((j\circ g)_+\cO_X)\cong \varphi_+\cO_{X\times V},
$$
and since we have $\varphi_+\cO_{X\times V}\in D_{rh}^{\leq 0}(\cD_{\widetilde{V}^\vee})$, we obtain
$\cH^1(\cR((j\circ g)_+)\cO_X)=0$. This implies that the
second triangle yields an exact sequence
$$
0 \longrightarrow \cH^{-1} \cR_{cst}((j\circ g)\cO_X)
  \longrightarrow \cH^0 \cR ((j\circ g)\cO_X)
  \longrightarrow \cH^0 \cR^\circ_c((j\circ g)\cO_X)
  \longrightarrow \cH^0 \cR_{cst}((j\circ g)\cO_X)
  \longrightarrow 0.
$$
Similarly to the proof of \cite[Theorem 2.1]{Reich2}, it can be shown that  $\cH^i \cR_{cst}((j\circ g)\cO_X)=
H^{n-1-i}(X,\dC)\otimes_\dC\cO_{\widetilde{V}^\vee}$ for $i=-1,0$. Moreover, we have seen above
that
$$
\begin{array}{c}
\cH^0 \cR^\circ_c((j\circ g)\cO_X) \cong \cH^0\FL((k\circ \iota)_+\cO_{\dG_m\times X})
\cong \cH^0\FL((k\circ \iota)_+\cO_{\dG_m}^{\beta_0}\boxtimes \cO_X) \\ \\
=\cH^0\FL(\check{\cM}(\widetilde{G},\widetilde{\rho},\widetilde{X},(\beta_0,0)))
=\cM(\widetilde{G},\widetilde{\rho},\widetilde{X},(\beta_0,0)),
\end{array}
$$
as required.
\end{proof}
Similarly to \cite[Proposition 3.1, Proposition 3.3.]{Reich2} it follows that this sequence
can be read in the category $\MHM(\widetilde{V}^\vee)$, where appropriate versions of the
Radon transformation functors can be defined.
We obtain the following consequence for the partial Fourier transformation
of the two (non trivial) $\cD$-modules in the above sequence.
\begin{corollary}\label{cor:IsoAfterLocFL}
For $\beta_0\in(-\infty,c)\cap \dZ$ we have
an isomorphism of $\cD_{\dA^1_z\times V^\vee}$-modules
$$
\FL^{\textup{loc}}_{V^\vee}(\cH^0 \varphi_+ \cO_{X\times V^\vee}) \cong \FL_{V^\vee}^{\textup{loc}}\cM(\widetilde{G},\widetilde{\rho},\widetilde{X},(\beta_0,0)).
$$
\end{corollary}
\begin{proof}
The functor $\FL_{V^\vee}^{\textup{loc}}$ is exact and kills
kernel and cokernel of the map
$$
\cH^0 \varphi_+ \cO_{X\times V^\vee} \longrightarrow \cM(\widetilde{G},\widetilde{\rho},\widetilde{X},(\beta_0,0))
$$
since these are are $\cO_{\widetilde{V}^\vee}$-locally free. This yields the statement of the corollary.
\end{proof}
\textbf{Remark:} Notice that it follows from our main result (Theorem \ref{theo:TautSystMHM}) that the partial Fourier transform
$\FL_{V^\vee}^{\textup{loc}}\cM(\widetilde{G},\widetilde{\rho},\widetilde{X},(\beta_0,0))$ underlies an irregular Hodge module
in the sense of \cite{Sa15}. However, since we do not have control over the Hodge filtration of $\cM(\widetilde{G},\widetilde{\rho},\widetilde{X},(\beta_0,0))$
for the moment, this structure cannot yet be entirely described.\\

Next we are going to consider the dimensional reduction of the tautological system $\cM(\widetilde{G}, \widetilde{\rho},\widetilde{X},(\beta_0,0))$.
As has been explained in the introduction, the main motivation to consider this operation is that it is parallel
to the reduction process from a GKZ-system to a classical
hypergeometric module that is considered in toric mirror symmetry (see, e.g. \cite[Section 3.1]{ReiSe} and \cite[Section 6]{ReiSe2}). As an example (which is covered
by the present case of a linear free divisor satisfying the (SK) condition but which is also of toric nature, i.e. which is a reduction of a GKZ-system to a classical hypergeometric module), consider the case where $D$ is the normal crossing divisor with $n$ components (the easiest example of a linear free divisor). Then the tautological system is a GKZ-system, more precisely, we have
$$
\cM(\widetilde{G}, \widetilde{\rho},\widetilde{X},(\beta_0,0))\cong
\frac{\cD_{\widetilde{V}}}{\left(\partial_{\lambda_0}^n-\prod_{i=1}^n
\partial_{\lambda_i}, \sum_{i=0}^n   \lambda_i\partial_{\lambda_i}+(n+1)+\beta_0,\left(\lambda_1\partial_{\lambda_1}-\lambda_i\partial_{\lambda_i}\right)_{i=2,\ldots,n}\right)}.
$$
We have the dual divisor  $D^\vee=\left\{h^\vee=\lambda_1\cdot\ldots\cdot\lambda_n=0\right\}$, and we can consider the morphism $\kappa:\dA^1_{\lambda_0}\times\dA^1_t\hookrightarrow \widetilde{V}$ given by $(\lambda_0,t)\mapsto (\lambda_0,t,1,\ldots,1)$. Then one calculates directly that the (non-characteristic) inverse image by $\kappa$ of
the localized GKZ-system is given as
\begin{equation}\label{eq:InverseImageReduction}
\kappa^+\left[\cM(\widetilde{G}, \widetilde{\rho},\widetilde{X},(\beta_0,0))\otimes_{\cO_{\widetilde{V}}} \cO_{\widetilde{V}^\vee}(*(\dA^1_{\lambda_0}\times D^\vee))\right]
\cong
\frac{\cD_{\dA^1_{\lambda_0\times\dA^1_t}}}{(t\partial_{\lambda_0}^n-(t\partial_t)^n,\lambda_0\partial_{\lambda_0}+nt\partial_t+(n+1)+\beta_0)}
\end{equation}
which corresponds, after a partial Fourier-Laplace transformation relative
to the parameter space $\dG_{m,t}$ to the quantum differential equations for the projective space $\dP^{n-1}$. The results below generalize this example
to the case of an arbitrary linear free divisor satisfying the (SK) condition. However, we will
consider a direct image to $\dA^1_{\lambda_0}\times \dA^1_t$ instead of the inverse image by $\kappa$ as above.

Consider again the equation $h^\vee$ of the dual divisor $D^\vee$,
seen as a morphism $h^\vee: V^\vee \rightarrow \dA^1_t$.
Let $\phi:=(\id_{\dA^1_{\lambda_0}},h^\vee):\widetilde{V}^\vee\rightarrow \dA^1_{\lambda_0}\times \dA^1_t$.
Then we have the following statement.
\begin{proposition}\label{prop:ReducTautSystem}
For any $\beta_0\in \dR$, write $\cM(*D^\vee)$ for the localization
$$
\cM(\widetilde{G},\widetilde{\rho},\widetilde{X},(\beta_0,0))\otimes_{\cO_{\widetilde{V}^\vee}} \cO_{\widetilde{V}^\vee}(*(\dA^1_{\lambda_0}\times D^\vee)).
$$
Then we have an isomorphism of $\cD_{\dA^1_{\lambda_0}\times \dG_{m,t}}$-modules
$$
\cH^0\phi_+(\cM(*D^\vee))
\cong \frac{\cD_{\dA^1_{\lambda_0}\times\dA^1_t}[t^{-1}]}{(\lambda_0\partial_{\lambda_0}+nt\partial_t+(n+1)+\beta_0, h(p)\cdot t\cdot\partial^n_{\lambda_0}- b_h(t\partial_t))}.
$$
\end{proposition}
Before starting the proof, we state the following preliminary lemma.
\begin{lemma}\label{lem:RestrDiffOp}
Let $X=\Spec(R)$, $Y=\Spec(T)$ two smooth affine algebraic varieties over $\dC$ and $g:X\rightarrow Y$ a surjective morphism
yielding an injective ring homomorphism $T\hookrightarrow R$. Consider the rings
of differential operators $D_R=\Gamma(X,\cD_X)$, $D_T=\Gamma(Y,\cD_T)$.\\
Let $P\in D_R$ be given, and suppose
that for all elements $t\in T$, we have $P(t)\in T$, where we see $P$ as an element of $End_\dC(R)$. Then $P$ yields
an element of $D_T$, that is, there exists an element $D_T$ which we denote by $P_{|T}$ such that for all $t\in T$ we have
$P(t)=P_{|T}(t)$. The order of $P_{|T}$ is smaller than or equal to the order of $P$.
\end{lemma}
\begin{proof}
This is elementary using Grothendiecks definition of $D_R$ resp. $D_T$, namely, the statement is obvious
if $P$ is a function (i.e., an element of $R$) or a vector field (i.e., an element of $\Der_\dC(R,R)$), and then one argues by induction on the degree of $P$.
\end{proof}
\begin{proof}[Proof of the proposition]
First note that according to the second statement of Lemma \ref{lem:CalculationTautSystem}, we have the following explicit expression of $\cM(*D^\vee)$:
\begin{equation}\label{eq:1}
\cM(*D^\vee) =
\frac{\cD_{\widetilde{V}^\vee}(*(\dA^1_{\lambda_0}\times D^\vee))}{\left(h(p)\partial_{\lambda_0}^n-h(\partial_{\lambda_1},\ldots,\partial_{\lambda_1}),
\delta^\vee_1,\ldots,\delta^\vee_{n-1},\sum_{i=0}^{n} \lambda_i\partial_{\lambda_i}+(n+1)+\beta_0\right)}
\end{equation}
where $\delta^\vee_1,\ldots,\delta^\vee_{n-1}$ is a basis of the module
$\Der(-\textup{log}(D^\vee))$ of vector fields on $V^\vee$ annihilating the equation $h^\vee$
of the dual divisor $D^\vee$ of $D$. Write more explicitely
$$
\delta^\vee_i = \sum_{j,k=1}^n \alpha^{(i)}_{jk} \lambda_j \partial_{\lambda_k}.
$$
for some $\alpha^{(i)}_{jk}\in \dC$.
Put $\cD:=\cD_{\widetilde{V}^\vee}(*(\dA^1_{\lambda_0}\times D^\vee))$ and consider the
right $\cD$-module $[\cM(*D^\vee)]^{\textup{right}}$ associated to $\cM(*D^\vee)$, which is given by
$\cD/(P_0,P_1,\ldots,P_n)\cD$, where
$$
P_0=h(p)\partial_{\lambda_0}^n-h(\partial_{\lambda_1},\ldots,\partial_{\lambda_n}), (P_i=\sum_{j,k=1}^n \alpha^{(i)}_{jk} \partial_{\lambda_k}\lambda_j )_{i=1,\ldots,n-1}, P_n=\sum_{i=0}^{n} \lambda_i\partial_{\lambda_i}-\beta_0.
$$
Notice that we have for all $i\in\{1,\ldots,n-1\}$ that
$$
P_i=\sum_{j,k=1}^n \alpha^{(i)}_{jk} \partial_{\lambda_k}\lambda_j
=\sum_{j,k=1}^n \alpha^{(i)}_{jk} \lambda_j\partial_{\lambda_k}+\textup{Trace}(\alpha^{(i)}_{jk} )
=\sum_{j,k=1}^n \alpha^{(i)}_{jk} \lambda_j\partial_{\lambda_k}=\delta_i^\vee
$$
since $\textup{Trace}(\alpha^{(i)}_{jk} )=0$ as  reductive linear free divisors are \emph{special} in the sense of \cite[Definition 2.1]{dGMS}.

Chose a $\cD$-free resolution $\cF^\bullet$ by right $\cD$-modules, i.e. an exact sequence
$$
\begin{tikzcd}
\ldots \ar{r} & \cD^{n+1} \ar{rr}{(P_0\cdot,\ldots,P_n\cdot)}&& \cD \ar{r} & \cM(*D^\vee) \ar{r}
& 0.
\end{tikzcd}
$$
Now consider the transfer module
$$
\cO_{\widetilde{V}^\vee}(*(\dA^1_{\lambda_0}\times D^\vee))
\otimes_{\phi^{-1}\cO_{\dA^1_{\lambda_0}\times \dA^1_t}[t^{-1}]} \phi^{-1}\cD_{\dA^1_{\lambda_0}\times \dA^1_t}[t^{-1}]
$$
which we abbreviate
by $\cD_\rightarrow$.
Recall that the left $\cD_{\widetilde{V}^\vee}(*(\dA^1_{\lambda_0}\times D^\vee))$-module structure on $\cD_\rightarrow$ is given as follows:
Interpret a section $g\otimes Q \in \cD_\rightarrow$ as a differential operator
from $\phi^{-1}\cO_{\dA^1_{\lambda_0}\times \dA^1_t}[t^{-1}]$ to $\cO_{\widetilde{V}^\vee}(*(\dA^1_{\lambda_0}\times D^\vee))$ sending $k\in\phi^{-1}\cO_{\dA^1_{\lambda_0}\times \dA^1_t}[t^{-1}]$ to
$g\cdot (\phi^*(Q(k)))$ (where $\phi^*:\phi^{-1}\cO_{\dA^1_{\lambda_0}\times \dA^1_t}[t^{-1}] \rightarrow
\cO_{\widetilde{V}^\vee}(*(\dA^1_{\lambda_0}\times D^\vee))$ is the morphism of sheaves of rings
that corresponds to $\phi$), then we have for all $P\in \cD_{\widetilde{V}^\vee}(*(\dA^1_{\lambda_0}\times D^\vee))$ that
$$
P(g\otimes Q)(k)=P(g\cdot \phi^*(Q(k))).
$$

The direct image complex $\phi_+ \cM(*D)$ is represented by the complex of left $\cD_{\dA^1_{\lambda_0}\times\dA^1_t}[t^{-1}]$-modules associated
to the complex of right $\cD_{\dA^1_{\lambda_0}\times\dA^1_t}[t^{-1}]$-modules $\phi_*(\cF\otimes_{\cD} \cD_\rightarrow)$ (using that $\phi$ is affine), where
$$
\begin{tikzcd}
\cF\otimes_{\cD_{\widetilde{V}^\vee}} \cD_\rightarrow: &  \ldots \ar{r} & \cD_\rightarrow^{n+1} \ar{r}{\Pi}& \cD_\rightarrow \ar{r}
& 0,
\end{tikzcd}
$$
where the last sheaf $\cD_\rightarrow$ sits in degree $0$ and where the map $\Pi$ is given by
$$
\Pi(g_0\otimes 1, \ldots, g_n\otimes 1)=\left[k\longmapsto\sum_{i=0}^{n} P_i(g_i\cdot\phi^*k)\right]
$$
for any $g_0,\ldots,g_n\in\cO_{\widetilde{V}^\vee}(*(\dA^1_{\lambda_0}\times D^\vee))$ (notice that because $\cD_\rightarrow$ is a right $\phi^{-1}\cD_{\dA^1_{\lambda_0}\times\dA^1_t}[t^{-1}]$-module, and the map $\Pi$ is $\phi^{-1}\cD_{\dA^1_{\lambda_0}\times\dA^1_t}[t^{-1}]$-linear,
it suffices to describe it on elements $g_i\otimes 1\in \cD_\rightarrow$).
Notice moreover that since $P_1,\ldots,P_{n-1}$ are vector fields and $P_n$ is a vector field plus a constant, we have
\begin{equation}\label{eq:LogHFields}
P_i(g_i\cdot(\phi^*k))=P_i(g_i)\cdot (\phi^*k)+g_i\cdot P_i(\phi^*k)
\end{equation}
for $i=1,\ldots,n-1$ and
\begin{equation}\label{eq:ActionEuler}
\begin{array}{rcl}
\D P_n(g_n\cdot(\phi^*k)) &= &
\D\left(\sum_{i=0}^{n} \lambda_i\partial_{\lambda_i}-\beta_0\right)(g_n)\cdot(\phi^*k)+
g_n\cdot\left(\sum_{i=0}^{n} \lambda_i\partial_{\lambda_i}\right)(\phi^*k)\\ \\
&=&
\D\left(\sum_{i=0}^{n} \lambda_i\partial_{\lambda_i}\right)(g_n)\cdot(\phi^*k)+
g_n\cdot\left(\sum_{i=0}^{n} \lambda_i\partial_{\lambda_i}-\beta_0\right)(\phi^*k).
\end{array}
\end{equation}
Our aim is to calculate the cohomology $\cH^0 \phi_*(\cF\otimes_{\cD} \cD_\rightarrow)$,
i.e., the cokernel of the map $\Pi$, seen as a $\cD_{\dA^1_{\lambda_0}\times \dA^1_t}[t^{-1}]$-module.

Notice that
for all $i\in\{1,\ldots,n-1\}$ and for all $k\in \phi^{-1}\cO_{\dA^1_{\lambda_0}\times \dA^1_t}[t^{-1}]$, we have $P_i(\phi^*k)=0$
since $P_i=\delta^\vee_i$ is a vector field in $\Der(-\log h^\vee)$.
Write $e_0,e_1,\ldots,e_{n-1},e_n$ for the canonical generators of $\cD_\rightarrow^{n+1}$, then we see
from from formula \eqref{eq:LogHFields} that
$$
\Pi \left((g_i\otimes 1)e_i\right)=
\delta^\vee_i(g_i)\otimes 1 \in \cD_\rightarrow,\ i=1,\ldots,n-1.
$$

In other words, the image of $\Pi$ is the right $\phi^{-1}\cD_{\dA^1_{\lambda_0}\times\dA^1_t}[t^{-1}]$-submodule
of $\cD_\rightarrow$ generated by
$$
\left\{\Pi((g_0\otimes 1)e_0), \delta_1^\vee(g_1)\otimes 1,\ldots,\delta_{n-1}^\vee(g_{n-1})\otimes 1,
\Pi((g_n\otimes 1)e_n)\,|\,g_0,\ldots,g_n\in
\cO_{\widetilde{V}^\vee}(*(\dA^1_{\lambda_0}\times D^\vee)) \right\}.
$$

Consider $\cO_{\widetilde{V}^\vee}(*(\dA^1_{\lambda_0}\times D^\vee))$ as
a $\phi^{-1}\cO_{\dA^1_{\lambda_0}\times\dA^1_t}[t^{-1}]$-module. Then it is clear that the
$\dC$-vector space
$$
\left\{\delta_i^\vee(g)\,|\, 1\leq i \leq n-1, g\in  \cO_{\widetilde{V}^\vee}(*(\dA^1_{\lambda_0}\times D^\vee)) \right\}
$$
has the structure of a
$\phi^{-1}\cO_{\dA^1_{\lambda_0}\times\dA^1_t}[t^{-1}]$-submodule (since
elements from $\phi^{-1}\cO_{\dA^1_{\lambda_0}\times\dA^1_t}[t^{-1}]$ are killed
by the vector fields $\delta_1^\vee, \ldots, \delta_{n-1}^\vee$).
We claim that we have an isomorphism of $\phi^{-1}\cO_{\dA^1_{\lambda_0}\times\dA^1_t}[t^{-1}]$-modules
$$
\frac{\cO_{\widetilde{V}^\vee}(*(\dA^1_{\lambda_0}\times D^\vee))}{\left\{\delta_i^\vee(g)\,|\, 1\leq i \leq n-1, g\in \cO_{\widetilde{V}^\vee}(*(\dA^1_{\lambda_0}\times D^\vee))\right\}}
\cong \phi^{-1}\cO_{\dA^1_{\lambda_0}\times\dA^1_t}[t^{-1}]
$$
or, equivalently (recall that the first component of $\phi$ is the identity) an isomorphism of $(h^\vee)^{-1}\cO_{\dA^1_t}[t^{-1}]$-modules
$$
\frac{\cO_{V^\vee}(*D^\vee)}{\left\{\delta_i^\vee(g)\,|\, 1\leq i \leq n-1, g\in \cO_{V^\vee}(*D^\vee))\right\}}
\cong (h^\vee)^{-1}\cO_{\dA^1_t}[t^{-1}].
$$
In order to show the claim, consider $n-1$-st (i.e. the top) cohomology of the relative (meromorphic) de Rham complex
$$
\cH^{n-1}(h^\vee)_*
(\Omega^\bullet_{V^\vee/\dA^1_t}(*D^\vee), d)
=
(h^\vee)_*\cH^{n-1}
(\Omega^\bullet_{V^\vee/\dA^1_t}(*D^\vee), d).
$$
The cohomology $(h^\vee)_*\cH^{n-1}
(\Omega^\bullet_{V^\vee/\dA^1_t}(*D^\vee), d)$ is nothing but $\cO_{\dA^1_t}[t^{-1}]$:
each (non-singular) fibre of
$h^\vee$ is an orbit of the dual action of $G=A_D$ on $V^\vee$, having finite stabilizers,
and since $A_D$ is reductive and connected, it has a deformation retraction to a compact connected $n-1$-dimensional
real Lie group, hence $H^{n-1}((h^\vee)^{-1}(t),\dC)=\dC$ for all $t\neq 0$.

Notice that we have
$$
(\Omega^\bullet_{V^\vee/\dA^1_t}(*D^\vee), d) \cong
(\Omega^\bullet(-\log h^\vee)(*D^\vee), d),
$$
where
$$
\D \Omega^\bullet(-\log h^\vee):=\frac{\Omega^\bullet_{V^\vee}(-\log D^\vee)}{\frac{d(h^\vee)}{h^\vee}\wedge\Omega^{\bullet-1}_{V^\vee}(-\log D^\vee)},
$$
see \cite[Section 2.2]{dGMS}. Then
$$
\begin{array}{rcl}
\D \cH^{n-1}
(\Omega^\bullet_{V^\vee/\dA^1_t}(*D^\vee), d)
&\cong&
\D\cH^{n-1}(\Omega^\bullet(-\log\,h^\vee)(*D^\vee),d)\\ \\
&\cong &\D\left(
\frac{\cO_{V^\vee}(*D^\vee)}{\left\{\delta_1^\vee(g),\ldots,\delta_{n-1}^\vee(g)\,|\,g\in \cO_{V^\vee}(*D^\vee)\right\}}
\right)\cdot \alpha
\end{array}
$$
where $\alpha=i_{\chi^\vee}(\vol/h^\vee)=n\vol/d(h^\vee)$ is a volume form in the fibres of $h^\vee$ (see \cite[Formula 2.7]{dGMS}). Here  $\chi^\vee$ denotes the Euler field $\sum_{i=1}^n \lambda_i\partial_{\lambda_i}$ in the space $V^\vee$
(Notice that we have again a decomposition $\Der(-\log D^\vee)=\Der(-\log h^\vee) \oplus \cO_{V^\vee} \chi^\vee$, where
$\Der(-\log h^\vee)=\{\theta\in\Der_{V^\vee}\,|\,\theta(h^\vee)=0\}$ since
$D^\vee$ is again a reductive linear free divisor),
and we write $i_{\chi^\vee}: \Omega^i_{V^\vee}(*D^\vee)
\rightarrow
\Omega^{i-1}_{V^\vee}(*D^\vee)$
for the interior derivative.

Namely, if we write $\lambda_j:=i_{\delta_j^\vee}(\alpha)\in\Omega^{n-2}(-\log\,h^\vee)$, then since
$d\lambda_j=0$ (because $D$ and $D^\vee$ are special, see \cite[Lemma 2.6]{dGMS}) and since $i_{\delta_j^\vee}(dg\wedge \alpha)=0\in\Omega^{n-1}(\log\,h^\vee)$ (see
\cite[Proof of Lemma 4.3]{dGMS}) the morphism
$d:\Omega^{n-2}(-\log\,h^\vee)\rightarrow \Omega^{n-1}(-\log\,h^\vee)$ is identified with
$$
\begin{array}{rcl}
\D\bigoplus_{j=1}^{n-1} \cO_{V^\vee} \lambda_j & \longrightarrow & \D\cO_{V^\vee} \alpha \\ \\
(g_1,\ldots,g_{n-1}) & \longmapsto & \left[\sum_{j=1}^{n-1} \delta_j^\vee(g_j)\right]\alpha.
\end{array}
$$

This shows the claim.
As a consequence, we have an identification
\begin{equation}\label{eq:BigIdent}
\begin{array}{c}
\D \frac{\cD_\rightarrow}{\left(\delta_1^\vee(g)\otimes 1,\ldots,\delta_{n-1}^\vee(g)\otimes 1,\,|\,g\in
\cO_{\widetilde{V}^\vee}(*(\dA^1_{\lambda_0}\times D^\vee)) \right)\phi^{-1}\cD_{\dA^1_{\lambda_0}\times \dA^1_t}[t^{-1}]}
=  \\ \\
\D \frac{\cO_{\widetilde{V}^\vee}(*(\dA^1_{\lambda_0}\times D^\vee))
\otimes_{\phi^{-1}\cO_{\dA^1_{\lambda_0}\times \dA^1_t}[t^{-1}]} \phi^{-1}\cD_{\dA^1_{\lambda_0}\times \dA^1_t}[t^{-1}]
}{\left(\delta_1^\vee(g)\otimes 1,\ldots,\delta_{n-1}^\vee(g)\otimes 1,\,|\,g\in
\cO_{\widetilde{V}^\vee}(*(\dA^1_{\lambda_0}\times D^\vee)) \right)\phi^{-1}\cD_{\dA^1_{\lambda_0}\times \dA^1_t}[t^{-1}]}
\\ \\
\D\cong
\phi^{-1}\cO_{\dA^1_{\lambda_0}\times \dA^1_t}[t^{-1}]
\otimes_{\phi^{-1}\cO_{\dA^1_{\lambda_0}\times \dA^1_t}[t^{-1}]} \phi^{-1}\cD_{\dA^1_{\lambda_0}\times \dA^1_t}[t^{-1}]
\cong \phi^{-1}\cD_{\dA^1_{\lambda_0}\times \dA^1_t}[t^{-1}]
\end{array}
\end{equation}
and hence
$$
\cD_\rightarrow/\im(\Pi) \cong
\frac{\phi^{-1}\cD_{\dA^1_{\lambda_0}\times \dA^1_t}[t^{-1}]}
{\left(\Pi((g_0\otimes 1)e_0), \Pi((g_n\otimes 1)e_n)\right)\phi^{-1}\cD_{\dA^1_{\lambda_0}\times \dA^1_t}[t^{-1}]}
$$
where now $\Pi((g_0\otimes 1)e_0)$ resp. $\Pi((g_n\otimes 1)e_n)$ denotes the image
of these two elements of $\cD_\rightarrow$ in $\phi^{-1}\cD_{\dA^1_{\lambda_0}\times \dA^1_t}[t^{-1}]$
under the identification given by equation \eqref{eq:BigIdent}.

Consider the Bernstein polynomial $b_{h^\vee}(s)=\prod_{i=1}^{n}(s-\alpha_i)$ of $h^\vee$
normalized such that we have
$$
h(\partial_{\lambda_1},\ldots,\partial_{\lambda_n})(h^\vee)^s=b_{h^\vee}(s)\cdot (h^\vee)^{s-1}.
$$
Notice (see \cite{granger_schulze_rims_2010}
or \cite{nar_symmetry_BS}) that the roots $\alpha_i$ are symmetric around zero
and write
$b_{h^\vee}(s)=s\cdot B_{h^\vee}(s)$, where we take the convention that
$B_{h^\vee}(s)=\prod_{i=2}^{n}(s-\alpha_i)$, i.e. that $\alpha_1=0$. We now claim that
\begin{equation}\label{eq:RmDirectImage}
\cH^0 \phi_*(\cF\otimes_{\cD} \cD_\rightarrow)\cong
\frac{\cD_{\dA^1_{\lambda_0}\times\dA^1_t}[t^{-1}]}{\left(\partial_{\lambda_0}^nh(p)-\partial_t\cdot B_{h^\vee}(t\partial_t), \partial_{\lambda_0}\lambda_0+nt\partial_t  -1-\beta_0\right)\cdot \cD_{\dA^1_{\lambda_0}\times\dA^1_t}[t^{-1}]}.
\end{equation}
Using formula \eqref{eq:ActionEuler}, we have
$$
P_n(\phi^*k)= (\sum_{i=0}^{n} \lambda_i\partial_{\lambda_i}-\beta_0)(\phi^*k)
=(\partial_{\lambda_0}\lambda_0+\sum_{i=1}^{n} \lambda_i\partial_{\lambda_i}-1-\beta_0)(\phi^*k),
$$
which means that the differential operator $\sum_{i=0}^{n} \lambda_i\partial_{\lambda_i}-\beta_0$ satisfies
the assumptions of the previous lemma (Lemma \ref{lem:RestrDiffOp}) for $X=\widetilde{V}^\vee$, $Y=\dA^1_{\lambda_0}\times \dA^1_t$
and the morphism $\phi:X\rightarrow Y$.
On the other hand, Bernstein's functional equation
$$
h(\partial_{\lambda_1},\ldots,\partial_{\lambda_n})(h^\vee)^s=b_{h^\vee}(s)\cdot (h^\vee)^{s-1}
$$
for the function $h^\vee$ implies that the differential operator
$h(\partial_1,\ldots,\partial_n)$ also satisfies the assumptions of the previous
lemma (in the situation where $X=\widetilde{V}^\vee$, $Y=\dA^1_{\lambda_0}\times\dA^1_t$, $\phi:X\rightarrow Y$).
Hence Lemma \ref{lem:RestrDiffOp} shows that they both define differential operators on the subalgebra $\phi^{-1}\cO_{\dA^1_{\lambda_0}\times \dA_t}$. Namely,
the operator
$h(\partial_{\lambda_1},\ldots,\partial_{\lambda_n})_{|\phi^{-1}\cO_{\dA^1_{\lambda_0}\times\dA^1_t}}\in \phi^{-1}(\cD_{\dA^1_{\lambda_0\times \dA^1_t}}[t^{-1}])$
corresponding to $h(\partial_{\lambda_1},\ldots,\partial_{\lambda_n})$ via the previous lemma
is precisely $\partial_t\cdot B_{h^\vee}(t\partial_t)$ (since it acts on $t^s$ as $h(\partial_{\lambda_1},\ldots,\partial_{\lambda_n})$ acts on $h^s$).  Similarly,
the operator $\sum_{i=0}^{n} \lambda_i\partial_{\lambda_i} \in \cD_{\widetilde{V}^\vee}$ corresponds
to $\lambda_0\partial_{\lambda_0}+nt\partial_t =\partial_{\lambda_0}\lambda_0+nt\partial_t-1 \in \phi^{-1}\cD_{\dA^1_{\lambda_0\times \dA^1_t}}[t^{-1}]$. This shows the claim, i.e. formula \eqref{eq:RmDirectImage}.

The final result follows
by taking the associated left $\cD_{\dA^1_{\lambda_0}\times\dA^1_t}[t^{-1}]$-module
of the right hand side of equation \eqref{eq:RmDirectImage}, notice that
$$
\begin{array}{rcl}
(\partial_t\cdot B_{h^\vee}(t\partial_t))^T&=&-B_{h^\vee}(t\partial_t)^T\cdot \partial_t
=\D (-1)^n\prod_{i=2}^{n}(t\partial_t-\alpha_i)^T\cdot \partial_t=(-1)^n\prod_{i=2}^{n}(-\partial_t t-\alpha_i)\cdot \partial_t=\\ \\
&=&  \D \prod_{i=1}^{n-1}\left(t\partial_t+1+\alpha_i\right)\partial_t
=  \partial_t\prod_{i=1}^{n-1}\left(t\partial_t+\alpha_i\right)
\stackrel{(*)}{=}  \partial_t\prod_{i=1}^{n-1}\left(t\partial_t-\alpha_i\right) \\ \\
& = &  \partial_tB_{h^\vee}(t\partial_t) = t^{-1} b_{h^\vee}(t\partial_t)
\end{array}
$$
where $(-)^T$ denotes the operation of taking the transpose operator and where
the equality (*) holds by the symmetry around $0$ of the roots of $B_{h^\vee}$.

Finally, as we have already noticed above, we can assume that $h$ and $h^\vee$ are equal
since both define linear free divisors, so that also $B_h=B_{h^\vee}$ resp. $b_h=b_{h^\vee}$.

\end{proof}

In the sequel, we draw some consequences of the above proposition.

\begin{corollary}\label{cor:DimReduc}
We have an isomorphism
$$
\FL_{\dG_{m,t}}^{\textup{loc}}(\cH^0(\phi\circ\varphi)_+\cO_{X\times V^\vee}(*(X\times D^\vee))) \cong
\frac{\cD_{\dA^1_z\times\dG_{m,t}}}{(z^n b_h(t\partial_t)-h(p)\cdot t,z^2\partial_z+ntz\partial_t)},
$$
where $\FL_{\dG_{m,t}}^{loc}: \textup{Mod}(\cD_{\dA^1_{\lambda_0}\times\dG_{m,t}})\rightarrow
\textup{Mod}(\cD_{\dA^1_{z}\times\dG_{m,t}})$ is the localized partial Fourier-Laplace transformation
with base $\dG_{m,t}$ (see formula \eqref{eq:FLloc} at the end of the introduction).
\end{corollary}
\begin{proof}
We deduce from corollary
\ref{cor:IsoAfterLocFL} that for any $\beta_0 \in (-\infty,c)\cap \dZ$ we have an isomorphism of $\cD_{\dA^1_z\times\dA^1_t}$-modules
$$
\cH^0(\id_{\dA^1_z},h^\vee)_+\left(\FL^{\textup{loc}}_{V^\vee}(\cH^0 \varphi_+ \cO_{X\times V^\vee})\right) \cong \cH^0(\id_{\dA^1_z},h^\vee)_+\left(\FL_{V^\vee}^{\textup{loc}}\cM(\widetilde{G},\widetilde{\rho},\widetilde{X},(\beta_0,0)))\right),
$$
and similarly we get
\begin{equation}\label{eq:GMEqualTautMerom}
\cH^0(\id_{\dA^1_z},h^\vee)_+\left(\FL^{\textup{loc}}_{V^\vee}(\cH^0 \varphi_+ \cO_{X\times V^\vee}(*(X\times D^\vee)))\right) \cong \cH^0(\id_{\dA^1_z},h^\vee)_+\left(\FL_{V^\vee}^{\textup{loc}}\cM(*D^\vee)\right)
\end{equation}
(assuming that $\beta_0$ used in the definition of $\cM(*D^\vee)$ satisfies $\beta_0 \in (-\infty,c)\cap \dZ$).
On the other hand, we have
$$
\cH^0(\id_{\dA^1_z},h^\vee)_+\FL^{\textup{loc}}_{V^\vee}(\cK)
\cong
\FL^{\textup{loc}}_{\dG_{m,t}} \cH^0 \phi_+ (\cK)
$$
for any $\cK\in D^b(\cD_{\widetilde{V}^\vee})$
since the first component of $\phi$ is the identity mapping on $\dA^1_{\lambda_0}$.
We know
that $\cH^i(\varphi_+ \cO_{X\times V^\vee}(*(X\times D^\vee)))$ is $\cO_{\widetilde{V}^\vee}(*(\dA^1_{\lambda_0}\times D^\vee))$-free of finite rank for $i<0$ since the restrictions $\varphi_{|X\times \{f\}}: X\rightarrow  \dA^1_{\lambda_0}\times\{f\}$ (for $f\in V^\vee\backslash D^\vee)$ are tame functions (see \cite[Section 3.3]{dGMS} for the tameness, and then \cite[Theorem 8.1]{Sa2}),
so that $\FL^{\textup{loc}}_{V^\vee}(\cH^i(\varphi_+ \cO_{X\times V^\vee}(*(X\times D^\vee))))=0$ for $i<0$. This implies that
$$
\begin{array}{rcl}
\cH^0(\id_{A^1_z},h^\vee)_+\FL^{\textup{loc}}_{V^\vee}(\varphi_+ \cO_{X\times V^\vee}(*(X\times D^\vee))) \cong
\cH^0(\id_{A^1_z},h^\vee)_+\FL^{\textup{loc}}_{V^\vee}(\cH^0\varphi_+ \cO_{X\times V^\vee}(*(X\times D^\vee))).
\end{array}
$$
Using this, it then follows from equation \ref{eq:GMEqualTautMerom} that
$$
\FL^{\textup{loc}}_{\dG_{m,t}} \cH^0 \phi_+ \varphi_+ \cO_{X\times V^\vee}(*(X\times D^\vee))
\cong
\FL^{\textup{loc}}_{\dG_{m,t}} \cH^0 \phi_+ \cM(*D^\vee).
$$
Recall that we have shown in Proposition \ref{prop:ReducTautSystem} that
$$
\begin{array}{rcl}
\D \cH^0\phi_+\cM(*D) & = &
\D \cH^0\phi_+\cM(\widetilde{G},\widetilde{\rho},\widetilde{X},(\beta_0,0))\otimes_{\cO_{\widetilde{V}^\vee}} \cO_{\widetilde{V}^\vee}(*(\dA^1_{\lambda_0}\times D^\vee)) \\ \\
&\cong& \D
\frac{\cD_{\dA^1_{\lambda_0}\times\dG_{m,t}}}{(\lambda_0\partial_{\lambda_0}+nt\partial_t+(n+1)+\beta_0,  h(p)\cdot t\cdot\partial^n_{\lambda_0}- b_h(t\partial_t))}.
\end{array}
$$
Now notice that
$$
\begin{array}{c}
\D \FL^{\textup{loc}}_{\dG_{m,t}}\left(\frac{\cD_{\dA^1_{\lambda_0}\times\dG_{m,t}}}{(\lambda_0\partial_{\lambda_0}+nt\partial_t+(n+1)+\beta_0, h(p)\cdot t\cdot\partial^n_{\lambda_0}- b_h(t\partial_t))}\right)
\\ \\ \D=\frac{\cD_{\dA^1_z\times\dG_{m,t}}}{(z^n b_h(t\partial_t)-h(p)\cdot t,z^2\partial_z+ntz\partial_t+z(n+\beta_0))},
\end{array}
$$
however, multiplication by $z$ is invertible on this module by construction (since it is a direct image
under the open embedding $j_z:\dG_{m,z}\times \dG_{m,t}\hookrightarrow \dA^1_z\times \dG_{m,t}$) and
it is easy to see that multiplication with $z^{n+\beta_0}$ induces an isomorphism
$$
\frac{\cD_{\dA^1_z\times\dG_{m,t}}}{(z^n b_h(t\partial_t)- h(p)\cdot t,z^2\partial_z+ntz\partial_t)}
\cong
\frac{\cD_{\dA^1_z\times\dG_{m,t}}}{(z^n b_h(t\partial_t)- h(p)\cdot t,z^2\partial_z+ntz\partial_t+z(n+\beta_0))}.
$$
\end{proof}
Next we discuss the relation of the $\cD$-modules obtained from tautological systems
associated to linear free divisors to the one studied in \cite{dGMS} and \cite{Sev1}.
Let $f\in (h^\vee)^{-1}(h(p))\subset V^\vee$ be a linear form on $V$. In these papers we have
considered the morphism $(f,h):V \longrightarrow \dA^1_s\times \dA^1_t$ and the direct image
of $\cO_V(*D)$ with respect to this morphism. Since this morphism depends on the chosen linear form $f$,
we would like to consider it here rather as a morphism
$$
\begin{array}{rcl}
\Psi:X^\vee\times V & \longrightarrow & \dA^1_s\times \dA^1_t \\ \\
(f,x) & \longmapsto & (f(x),h(x))
\end{array}
$$
where $X^\vee:=(h^\vee)^{-1}(h(p))$.
Then we have the following comparison result.
\begin{proposition}\label{prop:CalcGMSystemsLFD}
Let, as before, $p\in V\backslash D$, $X = h^{-1}(h(p))$ and $X^\vee=(h^\vee)^{-1}(h(p))$.
Then there is an isomorphism
in $D^b(\cD_{\dA^1_{\lambda_0}\times \dA^1_t})$
$$
\Psi_+ \cO_{X^\vee\times V}(*(X^\vee\times D)) \cong (\phi\circ \varphi)_+\cO_{X\times V^\vee}(*(X\times D^\vee))
$$
and hence an isomorphism of $\cD_{\dA^1_z\times \dG_{m,t}}$-modules
$$
\begin{array}{rcl}
\FL^{\textup{loc}}_{\dG_{m,t}}\cH^0\Psi_+\cO_{X^\vee\times V}(*(X^\vee\times D)) & \cong &
\FL^{\textup{loc}}_{\dG_{m,t}}(\cH^0(\phi\circ \varphi)_+(\cO_{X\times V^\vee}(*(X\times D^\vee)))) \\ \\
& \cong &
\frac{\D \cD_{\dA^1_z\times\dG_{m,t}}}{\D (z^n b_h(t\partial_t)-h(p)\cdot t,z^2\partial_z+ntz\partial_t)}.
\end{array}
$$
\end{proposition}
\begin{proof}
By choosing appropriate coordinates on $V$ (and the induced dual coordinates on $V^\vee$), we can assume
that the equations $h$ and $h^\vee$ are simply equal. This holds since for a a general reductive prehomogeneous vector space $V$
we have $h^\vee(\lambda) = \overline{h(\overline{w})}$ if $(w_1,\ldots,w_n)$ are unitary coordinates and $(\lambda_1,\ldots,\lambda_n)$ the corresponding dual coordinates (here $h$ is a defining equation of the
discriminant of $V$). But it is known
(see \cite[Theorem 2.5]{granger_schulze_rims_2010}) that for a linear free divisor $D$, its defining equation $h$
can be defined over $\dQ$ in appropriate coordinates. Consider the following diagram
$$
\begin{tikzcd}
X\times V^\vee \ar[hookrightarrow]{rr}{(g,\id_{V^\vee})} \ar[swap,bend right=40]{ddddrr}{\psi:=(can\circ(g,\id_{V^\vee}),h^\vee)}
\ar{rrdd}{\varphi}
&&
V \times V^\vee \ar{rr}{\textit{dual}} \ar[swap]{dd}{(can,pr)}
&& V^\vee \times V
\ar{lldddd}{(can',h)}& X^\vee\times V \ar[swap, hookrightarrow]{l} \ar[bend left=30]{llldddd}{\Psi} \\ \\
&& \dA^1_{\lambda_0}\times V^\vee =\widetilde{V}^\vee \ar[swap]{dd}{\SC\phi=(\id_{\lambda_0},h^\vee)} \\ \\
&& \dA^1_{\lambda_0}\times \dA^1_t
\end{tikzcd}
$$
here $can': V^\vee\times V \rightarrow \dA^1_{\lambda_0}$ is given by $(f,p)\mapsto f(p)$. On the other hand, we write $\textit{dual}$ for the morphism given by identifying $V$ with $V^\vee$ (and vice versa) via the chosen
coordinates $w_1,\ldots,w_n$ on $V$ and their dual coordinates $\lambda_1,\ldots,\lambda_n$ on $V^\vee$ (so it is not
just the involution reversing the factors of $V\times V^\vee$ resp. $V^\vee\times V$). Nevertheless, we have
$can=can'\circ \textit{dual}$. It follows that $(can,h^\vee)=(can',h)\circ \textit{dual}$ since $h$ is defined over $\dQ$.
In particular, the morphism $\textit{dual}$ sends $X\times V^\vee= h^{-1}(h(p))\times V^\vee$ isomorphically to
$(h^\vee)^{-1}(h(p))\times V=X^\vee\times V$. Similarly, the subvariety $X \times D^\vee$ inside $X\times V^\vee$ is
send to $X^\vee\times D$.

It is easy to check that the above diagram commutes.
We conclude that we have an isomorphism
$$
\begin{array}{c}
(\phi\circ \varphi)_+\cO_{X\times V^\vee}(*(X\times D^\vee))
=\psi_+ \cO_{X\times V^\vee}(*(X\times D^\vee)) \\ \\
=(can', h)_+(\textit{dual})_+(g,\id_{V^\vee})_+\cO_{X\times V^\vee}(*(X\times D^\vee))
\\ \\= ((can', h)\circ \textit{dual}_{|X\times V^\vee})_+\cO_{X\times V^\vee}(*(X\times D^\vee)) \cong \Psi_+ \cO_{X^\vee\times V}(*(X^\vee\times D)).
\end{array}
$$
The second assertion follows by combining this result with Corollary \ref{cor:DimReduc}.
\end{proof}
Notice that this gives exactly the result in \cite[Theorem 4]{Sev2}, which in turn was based
on the rather involved algorithmic arguments of \cite[section 4]{dGMS}. Actually, it is possible
to show Proposition \ref{prop:CalcGMSystemsLFD} without assuming the (SK) hypotheses. However,
since Theorem \ref{theo:Invertibility} is not available in this case, one is forced to consider a partial Fourier-Laplace transformation of
the object $(k'\circ\iota)_+\cN^{\beta_0}$ from formula
\eqref{eq:k_iota} instead of the total Fourier-Laplace transform of $\check{\cM}$, as has been done in the proof of Theorem \ref{theo:TautSystMHM}. The latter can be expressed
as a Radon transformation, but not the former, and hence the argument runs quite differently (compare also  \cite{SevCastReich} where a similar strategy is used in the toric case). We postpone this
discussion to a subsequent paper. \\

\textbf{Remark:} The most basic case of linear free divisors (satisfying the (SK) hypotheses) is the
normal crossing divisor given by $h=w_1\cdot\ldots\cdot w_n$. It is well known that in this case $G=A_D=\dG^{n-1}_m$,
and so the tautological system $\cM(G,\rho,X,\beta)$ is nothing but the GKZ-system $\cM_A^\beta$, where
$$
A=
\begin{pmatrix}
  1      & 0      & 0      & \ldots & 0      & -1 \\
  0      & 1      & 0      & \ldots & 0      & -1 \\
  \vdots & \vdots & \vdots & \ldots & \vdots & -1 \\
  0      & \ldots &        &        &  1     & -1
\end{pmatrix}.
$$
In this case the exact sequence of Proposition \ref{prop:4TermSequence} is the same as in
\cite[Theorem 2.13]{Reich2}, and obviously the reduced module $\FL_{\dG_{m,t}}^{\textup{loc}}\cH^0\phi_+(\cM(*D^\vee))$
(or rather its restriction to $z=1$) is nothing but the quantum differential equation of the projective space $\dP^{n-1}$.
Notice that in this case the dimensional reduction can be done by a direct image (under the map $\phi=(\id_{\dA^1_{\lambda_0}},h^\vee)$,
as in the current paper, as well as by a direct image under an embedding $\dA^1_{\lambda_0}\times\dG_{m,t}\hookrightarrow \dA^1_{\lambda_0}\times V^\vee$, as has
been done in \cite[Section 3.1]{ReiSe}. As we have mentioned at several places, it is a natural question to ask whether the more general tautological systems defined by prehomogeneous
group actions (say under the current hypotheses, i.e., with a linear free divisor satisfying (SK) as discriminant) can also
be interpreted as quantum differential equations of some variety or orbifold. This is particularly interesting in the case of quiver
discriminants, since one may hope to construct an appropriate A-model directly from the given quiver.

\bibliographystyle{amsalpha}

\begin{thebibliography}{CMNM09}

\bibitem[Ado94]{Adolphson}
Alan Adolphson, \emph{Hypergeometric functions and rings generated by
  monomials}, Duke Math. J. \textbf{73} (1994), no.~2, 269--290.

\bibitem[BHL{\etalchar{+}}14]{Taut2}
Spencer Bloch, An~Huang, Bong~H. Lian, Vasudevan Srinivas, and Shing-Tung Yau,
  \emph{On the holonomic rank problem}, J. Differential Geom. \textbf{97}
  (2014), no.~1, 11--35.

\bibitem[BM06]{BM}
Ragnar-Olaf Buchweitz and David Mond, \emph{Linear free divisors and quiver
  representations}, Singularities and computer algebra (Cambridge) (Christoph
  Lossen and Gerhard Pfister, eds.), London Math. Soc. Lecture Note Ser., vol.
  324, Cambridge Univ. Press, 2006, Papers from the conference held at the
  University of Kaiserslautern, Kaiserslautern, October 18--20, 2004,
  pp.~41--77.

\bibitem[BZMW18]{BMW}
Christine Berkesch~Zamaere, Laura~Felicia Matusevich, and Uli Walther, \emph{On
  normalized {H}orn systems}, preprint arXiv:1806.03355 [math.AG], 2018.

\bibitem[CDRS18]{SevCastReich}
Alberto Casta{\~n}o~Dom{\'\i}nguez, Thomas Reichelt, and Christian Sevenheck,
  \emph{Examples of hypergeometric twistor $\mathcal{D}$-modules}, Preprint
  arXiv:1803.04886, to appear in ``Algebra and Number Theory'', 2018.

\bibitem[CDS17]{SevCast}
Alberto Casta{\~n}o~Dom{\'\i}nguez and Christian Sevenheck, \emph{Irregular
  {H}odge filtration of some confluent hypergeometric systems}, Preprint
  arXiv:1707.03259, to appear in ``Journal de l'Institut de math\'ematiques de
  Jussieu'', available at \url{https://doi.org/10.1017/S1474748019000288},
  2017.

\bibitem[CJUE04]{cas_ucha_exper}
Francisco~J. Castro~Jim{\'e}nez and Jos{\'{e}}~M. Ucha~Enr{\'{\i}}quez,
  \emph{Testing the logarithmic comparison theorem for free divisors},
  Experiment. Math. \textbf{13} (2004), no.~4, 441--449.

\bibitem[CM99]{calde_ens}
Francisco~J. Calder\'on~Moreno, \emph{Logarithmic differential operators and
  logarithmic de {R}ham complexes relative to a free divisor}, Ann. Sci.
  \'Ecole Norm. Sup. (4) \textbf{32} (1999), no.~5, 701--714.

\bibitem[CMNM02]{calde_nar_compo}
Francisco~J. Calder\'on~Moreno and Luis Narv\'aez~Macarro, \emph{The module
  {$\mathcal{D}f^s$} for locally quasi-homogeneous free divisors}, Compositio
  Math. \textbf{134} (2002), no.~1, 59--74.

\bibitem[CMNM09]{calde_nar_lct_ilc}
\bysame, \emph{On the logarithmic comparison theorem for integrable logarithmic
  connections.}, Proc. London Math. Soc. \textbf{98} (2009), no.~3, 585--606.

\bibitem[DS13]{DettSa}
Michael Dettweiler and Claude Sabbah, \emph{Hodge theory of the middle
  convolution}, Publ. Res. Inst. Math. Sci. \textbf{49} (2013), no.~4,
  761--800.

\bibitem[Giv98]{Giv7}
Alexander Givental, \emph{A mirror theorem for toric complete intersections},
  Topological field theory, primitive forms and related topics (Kyoto, 1996),
  Progr. Math., vol. 160, Birkh\"auser Boston, Boston, MA, 1998, pp.~141--175.

\bibitem[GKZ90]{GKZ1}
Israel~M. Gel{\cprime}fand, Mikhail~M. Kapranov, and Andrei~V. Zelevinsky,
  \emph{Generalized {E}uler integrals and {$A$}-hypergeometric functions}, Adv.
  Math. \textbf{84} (1990), no.~2, 255--271.

\bibitem[GMNS09]{GMNS}
Michel Granger, David Mond, Alicia Nieto, and Mathias Schulze, \emph{Linear
  free divisors and the global logarithmic comparison theorem.}, Ann. Inst.
  Fourier (Grenoble) \textbf{59} (2009), no.~1, 811--850.

\bibitem[GMS09]{dGMS}
Ignacio~de Gregorio, David Mond, and Christian Sevenheck, \emph{Linear free
  divisors and {F}robenius manifolds}, Compositio Mathematica \textbf{145}
  (2009), no.~5, 1305--1350.

\bibitem[GS10]{granger_schulze_rims_2010}
Michel Granger and Mathias Schulze, \emph{On the symmetry of {$b$}-functions of
  linear free divisors}, Publ. Res. Inst. Math. Sci. \textbf{46} (2010), no.~3,
  479--506. \MR{2760735}

\bibitem[Hot98]{HottaEq}
Ryoshi Hotta, \emph{Equivariant $\mathcal{D}$-modules}, Preprint
  math.RT/9805021, 1998.

\bibitem[Iri09]{Ir2}
Hiroshi Iritani, \emph{{A}n integral structure in quantum cohomology and mirror
  symmetry for toric orbifolds}, Adv. Math. \textbf{222} (2009), no.~3,
  1016--1079.

\bibitem[Kap98]{KapranovReduct}
Mikhail Kapranov, \emph{Hypergeometric functions on reductive groups},
  Integrable systems and algebraic geometry ({K}obe/{K}yoto, 1997) (M.-H.
  Saito, Y.~Shimizu, and K.~Ueno, eds.), World Sci. Publ., River Edge, NJ,
  1998, pp.~236--281.

\bibitem[Kim03]{Kim}
Tatsuo Kimura, \emph{Introduction to prehomogeneous vector spaces},
  Translations of Mathematical Monographs, vol. 215, American Mathematical
  Society, Providence, RI, 2003, Translated from the 1998 Japanese original by
  Makoto Nagura and Tsuyoshi Niitani and revised by the author.

\bibitem[LSY13]{Taut1}
Bong~H. Lian, Ruifang Song, and Shing-Tung Yau, \emph{Periodic integrals and
  tautological systems}, J. Eur. Math. Soc. (JEMS) \textbf{15} (2013), no.~4,
  1457--1483.

\bibitem[LY13]{Taut3}
Bong~H. Lian and Shing-Tung Yau, \emph{Period integrals of {CY} and general
  type complete intersections}, Invent. Math. \textbf{191} (2013), no.~1,
  35--89.

\bibitem[MP16]{PopaMustata}
Mircea Musta\c{t}\u{a} and Minhea Popa, \emph{Hodge ideals}, Preprint
  math.AG/1605.08088, to appear in ``Memoirs of the AMS'', 2016.

\bibitem[NM15]{nar_symmetry_BS}
Luis Narv\'aez~Macarro, \emph{A duality approach to the symmetry of
  {B}ernstein-{S}ato polynomials of free divisors}, Adv. Math. \textbf{281}
  (2015), 1242--1273.

\bibitem[Pop18]{PopaICMTalk}
Minhea Popa, \emph{D-modules in birational geometry}, Preprint
  math.AG/1807.02375, to appear in ``Proceedings of the ICM, Rio de Janeiro'',
  2018.

\bibitem[{Rei}14]{Reich2}
Thomas {Reichelt}, \emph{{Laurent Polynomials, GKZ-hypergeometric Systems and
  Mixed Hodge Modules}}, Compositio Mathematica \textbf{(150)} (2014),
  911--941.

\bibitem[Rin63]{Rinehart-1963}
George~S. Rinehart, \emph{Differential forms on general commutative algebras},
  Trans. Amer. Math. Soc. \textbf{108} (1963), 195--222.

\bibitem[RS15a]{ReiSe3}
Thomas Reichelt and Christian Sevenheck, \emph{Hypergeometric {H}odge modules},
  Preprint math.AG/1503.01004, to appear in ``Algebraic Geometry'', 2015.

\bibitem[RS15b]{ReiSe}
\bysame, \emph{Logarithmic {F}robenius manifolds, hypergeometric systems and
  quantum $\mathcal{D}$-modules}, Journal of Algebraic Geometry \textbf{24}
  (2015), no.~2, 201--281.

\bibitem[RS17]{ReiSe2}
\bysame, \emph{Non-affine {L}andau-{G}inzburg models and intersection
  cohomology}, Ann. Sci. \'Ec. Norm. Sup\'er. (4) \textbf{50} (2017), no.~3,
  665--753.

\bibitem[RSW18]{ReichSevWalth}
Thomas Reichelt, Christian Sevenheck, and Uli Walther, \emph{On the
  {$b$}-functions of hypergeometric systems}, Int. Math. Res. Not. IMRN (2018),
  no.~21, 6535--6555.

\bibitem[Sab06]{Sa2}
Claude Sabbah, \emph{Hypergeometric periods for a tame polynomial}, Port. Math.
  (N.S.) \textbf{63} (2006), no.~2, 173--226, written in 1998.

\bibitem[Sab18]{Sa15}
\bysame, \emph{Irregular {H}odge theory}, Mém. Soc. Mat. Fr. (N.S.)
  \textbf{156} (2018), available at arXiv:1511.00176 [math.AG] (with the
  collaboration of Jeng--Daw Yu).

\bibitem[Sai80]{KS1}
Kyoji Saito, \emph{Theory of logarithmic differential forms and logarithmic
  vector fields}, J. Fac. Sci. Univ. Tokyo Sect. IA Math. \textbf{27} (1980),
  no.~2, 265--291.

\bibitem[Sev11]{Sev1}
Christian Sevenheck, \emph{Bernstein polynomials and spectral numbers for
  linear free divisors}, Ann. Inst. Fourier (Grenoble) \textbf{61} (2011),
  no.~1, 379--400.

\bibitem[Sev13]{Sev2}
\bysame, \emph{Duality of {G}au\ss -{M}anin systems associated to linear free
  divisors}, Math. Z. \textbf{274} (2013), no.~1-2, 249--261. \MR{3054328}

\bibitem[SW09]{SchulWalth2}
Mathias Schulze and Uli Walther, \emph{Hypergeometric $\mathcal{D}$-modules and
  twisted {G}au\ss-{M}anin systems}, J. Algebra \textbf{322} (2009), no.~9,
  3392--3409.

\end{thebibliography}

\newcommand{\etalchar}[1]{$^{#1}$}
\def\cprime{$'$}
\providecommand{\bysame}{\leavevmode\hbox to3em{\hrulefill}\thinspace}
\providecommand{\MR}{\relax\ifhmode\unskip\space\fi MR }
\providecommand{\MRhref}[2]{  \href{http://www.ams.org/mathscinet-getitem?mr=#1}{#2}
}
\providecommand{\href}[2]{#2}

\vspace*{1cm}

\nd
Luis Narv\'{a}ez Macarro \\
Departamento de \'{A}lgebra \& Instituto de Matem\'{a}ticas (IMUS)\\
Facultad de Matem\'{a}ticas\\
Universidad de Sevilla\\
41080 Sevilla\\
Spain

narvaez@us.es
\vspace*{1cm}

\nd
Christian Sevenheck\\
Fakult\"at f\"ur Mathematik\\
Technische Universit\"at Chemnitz\\
09107 Chemnitz\\
Germany\\
christian.sevenheck@mathematik.tu-chemnitz.de

\end{document}